\documentclass[11pt]{article}
\usepackage{amsmath}
\usepackage{amssymb}
\usepackage{array}
	\newcolumntype{C}[1]{>{\centering}m{#1}}

\usepackage{bbm}
\usepackage{booktabs}
\usepackage[round]{natbib}

\usepackage{upgreek}
\usepackage[margin=1in]{geometry}

\usepackage{./header-files/marginalia}

\usepackage{xkvltxp}
\usepackage{fixme}

\usepackage{stmaryrd}

\usepackage{hyperref}

\usepackage{multirow}

\usepackage{listings}
\usepackage{xargs}
\usepackage{xstring}

\usepackage{multicol}

\usepackage{graphicx}

\usepackage{float}
\usepackage{color}
\usepackage{enumitem}

\newcommand{\EE}{\mathbb{E}}

\newcommand{\NN}{\mathbb{N}}

\newcommand{\PP}{\mathbb{P}}

\newcommand{\RR}{\mathbb{R}}

\newcommand{\One}{\mathbbm{1}}

\newcommand{\one}{\mathbf{1}}

\newcommand{\Dd}{\mathcal{D}}
\newcommand{\Ee}{\mathcal{E}}
\newcommand{\Ff}{\mathcal{F}}

\newcommand{\Hh}{\mathcal{H}}

\newcommand{\Ll}{\mathcal{L}}

\newcommand{\Nn}{\mathcal{N}}

\newcommand{\vx}{\mathbf{x}}

\newcommand{\vB}{\mathbf{B}}

\newcommand{\vX}{\mathbf{X}}
\newcommand{\vY}{\mathbf{Y}}

\renewcommand{\Pr}{\mathbb{P}} % probability
\def\EE{\mathbb{E}} % expectation 

\newcommand{\EEE}[1]{\underset{#1}{\EE}}
\newcommand{\Prr}[1]{\underset{#1}{\Pr}}

\def\[#1\]{\begin{equation}\begin{aligned}#1\end{aligned}\end{equation}}

\def\*[#1\]{\begin{align*}#1\end{align*}}

\newcommand{\st}{\ \text{s.t.}\ }
\newcommand{\andT}{\ \text{and}\ }

\DeclareMathOperator*{\Cov}{Cov}
\DeclareMathOperator*{\Corr}{Corr}
\DeclareMathOperator*{\Var}{Var}

\newcommand{\VVar}[1]{\underset{#1}{\Var}}

\DeclareMathOperator*{\argmax}{\arg\max}
\DeclareMathOperator*{\esssup}{\text{ess}\sup}
\DeclareMathOperator*{\essinf}{\text{ess}\inf}

\newcommand{\lcr}[3]{\left #1 #2 \right #3}
\newcommand{\lcrx}[4][{-1}]{
	\IfEq{#1}{-1}{\left #2 {{{{#3}}}} \right #4}{
   	\IfEq{#1}{0}{#1 {{{{#2}}}} #3}{
	\IfEq{#1}{1}{\bigl #2 {{{{#3}}}} \bigr #4}{
	\IfEq{#1}{2}{\Bigl #2 {{{{#3}}}} \Bigr #4}{
	\IfEq{#1}{3}{\biggl #2 {{{{#3}}}} \biggr #4}{
	\IfEq{#1}{4}{\Biggl #2 {{{{#3}}}} \Biggr #4}{
    \GenericWarning{"4th argument to lcrx must be -1, 0, 1, 2, 3, or 4"}
    }}}}}}}

\newcommand{\abs}[2][{-1}]{\lcrx[#1] \vert {#2} \vert }
\newcommand{\set}[2][{-1}]{\lcrx[#1] \{ {#2} \}}

\newcommand{\ceil}[2][{-1}]{\lcrx[#1] \lceil {#2} \rceil}
\newcommand{\norm}[2][{-1}]{\lcrx[#1] \Vert {#2} \Vert}
\newcommand{\inner}[3][{-1}]{\lcrx[#1] \langle {{#2},\ {#3}} \rangle}

\newcommand{\myeq}[1]{\ensuremath{\stackrel{\text{#1}}{=}}}
\newcommand{\myeqm}[1]{\ensuremath{\stackrel{#1}{=}}}

\newcommand{\myram}[1]{\ensuremath{\stackrel{#1}{\to}}}

\newcommand{\mysim}[1]{\ensuremath{\stackrel{\text{#1}}{\sim}}}

\newcommand{\ol}{\overline}
\newcommand{\ul}{\underline}
\newcommand{\mb}{\mathbf}

\newcommandx{\dd}[4][1, 2=\mathrm{d}, 4]{\ensuremath{\frac{#2^{#4}#1}{#2#3^{#4}}}}
\newcommandx{\pp}[4][1, 2=\partial, 4]{\ensuremath{\frac{#2^{#4}#1}{#2#3^{#4}}}}

\DeclareMathOperator*{\newlim}{\mathrm{lim}\vphantom{\mathrm{infsup}}}
\DeclareMathOperator*{\newmin}{\mathrm{min}\vphantom{\mathrm{infsup}}}
\DeclareMathOperator*{\newmax}{\mathrm{max}\vphantom{\mathrm{infsup}}}
\DeclareMathOperator*{\newinf}{\mathrm{inf}\vphantom{\mathrm{infsup}}}
\DeclareMathOperator*{\newsup}{\mathrm{sup}\vphantom{\mathrm{infsup}}}
\renewcommand{\lim}{\newlim}
\renewcommand{\min}{\newmin}
\renewcommand{\max}{\newmax}
\renewcommand{\inf}{\newinf}
\renewcommand{\sup}{\newsup}

\newcommand{\grad}{\nabla}
\newcommand{\laplace}{\Delta}

\usepackage{amsmath,amsthm,amssymb}

\newtheorem{theorem}{Theorem}
\newtheorem{proposition}[theorem]{Proposition}
\newtheorem{corollary}[theorem]{Corollary}
\newtheorem{lemma}[theorem]{Lemma}
\newtheorem{definition}[theorem]{Definition}
\numberwithin{theorem}{section}
\newtheorem*{assumption*}{Assumption}

\theoremstyle{remark}
\newtheorem{remark}[theorem]{Remark}

\usepackage{mathtools}
\def\multiset#1#2{\ensuremath{\left(\kern-.3em\left(\genfrac{}{}{0pt}{}{#1}{#2}\right)\kern-.3em\right)}}

\newcommand{\tr}{\mathrm{Tr}}
\newcommand{\vol}{\mathrm{Vol}}

\newcommand{\given}{\vert}

\DeclareMathOperator*{\divergence}{div}
\DeclareMathOperator{\supp}{Support}

\newcommand{\poisson}{\mathrm{Poisson}}

\makeatletter
\newcommand{\subalign}[1]{%
  \vcenter{%
    \Let@ \restore@math@cr \default@tag
    \baselineskip\fontdimen10 \scriptfont\tw@
    \advance\baselineskip\fontdimen12 \scriptfont\tw@
    \lineskip\thr@@\fontdimen8 \scriptfont\thr@@
    \lineskiplimit\lineskip
    \ialign{\hfil$\m@th\scriptstyle##$&$\m@th\scriptstyle{}##$\crcr
      #1\crcr
    }%
  }
}
\makeatother

\newenvironment{myproof}[1][\proofname]{%
  \begin{proof}[#1]$ $\par\nobreak\ignorespaces
}{%
  \end{proof}
}

\newcommand{\distto}{\rightsquigarrow}
\newcommand{\disttosk}{\Rightarrow}

\DeclareMathOperator{\spec}{\sigma}
\DeclareMathOperator{\diag}{diag}

\newcommand{\NNO}{{\NN\cup\set{0}}}

\title{Optimal Scaling and Shaping of Random Walk Metropolis via Diffusion Limits of Block-I.I.D. Targets }
\author{Jeffrey Negrea \footnote{University of Toronto -- \href{mailto:negrea@stat.toronto.edu}{negrea@stat.toronto.edu}-- Supported by a Vanier Canada Graduate Scholarship}}
\begin{document}

\maketitle
\abstract{This work extends \citep{roberts1997weak} by consid\-ering limits of Random Walk Metropolis (RWM) applied to block IID target distributions, with corresponding block-independent propos\-als. The extension verifies the robustness of the optimal scaling heuristic, to tune the acceptance rate to $\approx0.234$, for any choice of proposal shaping. We upgrade the form of weak convergence from a finite-dimensional subprocess to the infinite dimensional process. We show that the optimal shaping (in terms of the decay of autocorrelations of linear functions) is the variance of the target distribution. We show that this choice coincides with the optimal shaping in terms of spectral gaps in special cases where they can be computed. Lastly, we provide some negative guarantees, showing that RWM perf\-ormance degrades with higher-order dependence. In such cases, no tuning of RWM will yield performance comparable to an IID target.

}
\section{Introduction}
%!TEX root = ../blockiid_main.tex
\setlength{\parindent}{0pt}
\setlength{\parskip}{6pt}

\subsection{The optimal scaling problem \& diffusion limits}
Markov Chain Monte Carlo (MCMC) algorithms are a common tool for estimating expectations with respect to a arbitrary ``target'' probability measures. These methods operate by defining a Markov Chain whose stationary distribution is the target, and whose dynamics are easily computable. Running this Markov chain forwards in time yields a dependent sequence of samples which can be used to estimate expectations. Performance of such algorithms are typically measured based on how quickly empirical expectations will converge to their target values. Among the simplest of such algorithms is Random-Walk Metropolis (RWM), which proposes IID increments (from a ``proposal distribution'') which are either accepted or rejected with probabilities tuned to match the target stationary distribution. Proposals which land in areas with low target density are likely to be rejected, while those that land in areas with higher density are likely to be accepted. 

The choice of proposal distribution is they key tuning parameter in the design and application of RWM algorithms and has a decisive impact on the performance of the algorithm, especially in a high dimensional setting. A typical choice is to use a mean-zero Gaussian proposal, yet among this class on is still required to select the variance-covariance matrix of the proposal. Proposing steps which are too large in any particular direction will lead to poor performance via frequent rejection, as the proposed point will typically have low target density. Proposing steps which are too small in any particular direction will lead to poor performance, since it will take many steps to move a meaningful distance in that direction. A step size and orientation distribution which is ``just right" (not too big, and also not too small) is required for good performance. This leads us to consider the optimal shaping and scaling for Gaussian proposals for the RWM algorithm.

The seminal paper of \cite{roberts1997weak} introduced techniques for ana\-lysing the optimal scaling problem in the limit, as the dimension of the target tends to infinity, for independent and identically distributed targets (IID targets). The key insight was that (under appropriate rescaling) the random paths of any single component converge in law to the random path of a diffusion process (more precisely, weak convergence in the Skorohod topology), and that the speed of the limiting diffusion can be optimised using elementary techniques. Since that work, there has been a reasonable amount of attention placed on extending their results to other MCMC algorithms, as well as to more general targets. 

The main novel contributions of this work are (i) we consider block-inde\-pendent targets with possibly large dependence structures within blocks, (ii) we consider non-spherical proposals, (iii) we show that the random path of an full dependent block converges in law to the path of a multivariate anisotropic diffusion, (iv) we show that the entire random path in $\RR^\infty$ converges in law to an infinite product of independent multivariate aniso\-tropic diffusions, (v) we address both the scaling and the shaping of the proposal under joint convergence, and (vi) we interpret our results to provide conditions under which high-dimensional dependence in the target distri\-bution will cause RWM performance under optimal shaping and scaling to deteriorate relative to any I.I.D. target. 

Some additional contributions of this work, which are not the main focus, are that (a) we have relaxed the assumptions needed to prove weak convergence in the Skorohod topology of the process even in the special case of I.I.D. targets (this amounts to fewer moment conditions and smoothness cond\-itions on the target distribution), (b) we have proven a very general inte\-gration by parts lemma for probability distributions which may be of interest more broadly than this work (for example in the context of Stein's Method), and (c) we give a fairly comprehensive list of consequences of the assumption that a probability density $\pi$ has $\grad\log\pi$ Lipschitz.

\subsection{Outline of this work}
A brief, non-comprehensive summary of existing work in the area is given in Subsection \ref{subsec:previous-work}. We provide notation and defin\-itions used through out the paper in Section \ref{subsec:notation}, which sets up the weak convergence and optimal scaling/shaping problems. That subsection also provides a fairly comprehensive list of consequences of the only assumption used to prove weak convergence (that, when $\pi$ is the target density, $\grad\log\pi$ is Lipschitz).

In Section \ref{sec:keyresults} we offer the main contributions of this work. Subsection \ref{subsec:weak} states the weak convergence result for finite dimensional processes upon which everything is built and states and proves the convergence result for the infinite dimensional process. Subsection \ref{subsec:scaling} provides the optimal scaling of the proposal for a fixed shaping. Subsection \ref{subsec:shaping1} presents the optimal shaping in terms of the spectral gap of the generator for certain special target distributions for which it is analytically tractable. Subsection \ref{subsec:shaping2} presents the optimal shaping in terms of short term autocorrelations for more general target distributions, and demonstrates that this alternative objective upper bounds the spectral gap, providing ``speed limits'' on the performance of RWM algorithms. Lastly among the key results, Subsection \ref{subsec:hd-dependence} provides a discursive analysis of the implications of the derived speed limits upon the performance decay of RWM in scenarios of high-dimensional dependence, relative to the independent target case.

Section \ref{sec:pf-weak-conv} provides the proof of weak convergence in the Skorohod topology for finite dim\-ensional processes. Section \ref{sec:ibps}, which may also be of interest to readers not working in MCMC theory or practice, proves an integration by parts lemma for probability distribution (in Sub\-section \ref{subsec:ibps}) which is more general than what we have found in the literature, and proves several consequences of the assumption that a density, $\pi$ has $\grad\log\pi$ Lipschitz (in Subsection \ref{subsec:gll-dens}).

\subsection{Previous work}\label{subsec:previous-work}
The seminal work utilising diffusion limits and weak convergence in the Skorohod topology to address the optimal scaling problem in MCMC is \citep{roberts1997weak}. That work considers I.I.D. targets of the form $\pi^{\otimes d}$ as $d\to\infty$, where $\pi$ is a density on $\RR^1$ with $D\log\pi$ Lipschitz continuous, and the MCMC algorithm is RWM with spherical Gaussian proposals. The paper has additional regularity assumptions of smoothness (that the density is twice continuously differentiable, though the proof given actually also uses the existence of bounded third der\-ivatives) and moment conditions (that $\EEE{X\sim\pi} (D\log\pi(X))^8 <\infty$ and $\EEE{X\sim\pi} (\frac{D^2\pi}{\pi})^4 <\infty$). That work proves weak convergence of the first component's path process to that of a univariate Langevin diffusion, and derives an optimal scaling criteria of accepting $\approx 23.4 \% $ of proposals for the RWM algorithm with I.I.D targets. The goal of our present work can be summarised as extending the results of \citep{roberts1997weak}.

The follow-up paper, \citep{roberts1998optimal}, derives similar optimal scaling results for the Metropolis Adjusted Langevin Algorithm (MALA). That work considers I.I.D. targets of the form $\pi^{\otimes d}$ as $d\to\infty$ as well, where $\pi$ is a density on $\RR^1$ with $D\log\pi$ Lipschitz continuous, and the MCMC algorithm is MALA with an isotropic diffusion term. The paper has additional regularity assumptions of smoothness (that the density is eight times continuously differentiable), a growth assumption (that the first eight derivatives of $\log\pi$ are all bounded by a polynomial) and moment conditions (that all polynomial moments are finite; $\forall k\in\NN \lcr({\EEE{X\sim\pi}(X^k)<\infty })$. That work proves weak convergence of the first component's path process to that of a univariate Langevin diffusion, and derives an optimal scaling criteria of accepting $\approx 57.4  \% $ of proposals for the MALA algorithm with I.I.D targets. We do not presently extend this article, but may apply the techniques of this paper to MALA in subsequent work

The survey paper, \citep{roberts2001optimal}, provides further context to the optimal scaling problem and presents theoretical and empirical results clear\-ly and concisely. As well as summ\-arising previous work, the paper provides an examination of how the optimal scaling in finite dimensions approaches the infinite dimensional limits derived via diffusion limits. Lastly, and a large inspiration for this work, \citep{roberts2001optimal} consider extensions to independent products which differ (only) by heterogeneity of scale, provide the first optimal shaping result. They note that ``This result does not appear in any of the MCMC scaling literature, so we have sketched a proof which appears in the Appendix,'' however the sketched proof considers only conv\-ergence of a single component and so the impact of optimal scaling and shaping on the mixing properties of the full multidimensional target may be questioned. Our present paper builds on these ideas to provide multivariate convergence and optimal scaling and shaping results which apply to the full multi\-dimensional limit.

The work of \citep{neal2006optimal} considers modified RWM and MALA algo\-rithms where only a fraction of the components are updated at a time. 
Algorithms of that type are typically more efficient as the an update which would have been rejected because of a single ``bad proposal'' in one comp\-onent is not going to affect the speed of all dimensions. 
That paper derives the optimal scaling and update rate simultaneously with the same assum\-ptions as \citep{roberts1997weak}. 
In Section \ref{sec:pf-weak-conv} we borrow the structure of the weak convergence proof from the detailed and precise description of\citep{neal2006optimal}. 

The papers \citep{bedard2007weak} and \citep{bedard2008optimal}, and the related Ph.D. thesis \citep{bedard2006robustness} consider a more extreme version of the scale homogeneity problem for the RWM algorithm. Particularly they address the case that the scaling of various components shrink or grow at disparate rates as the dimension tends to infinity. That collection of work shows that, depending on which scalings are dominant, the limiting law of the first component may be either a univariate RWM process or a Langevin diffusion, and that in certain situations the optimal acceptance rate will be quite different than the $23.4\% $ of the homogeneous or limited inhomogeneity cases. That work also slightly relaxed the assumptions of the original paper of \citep{roberts1997weak} by reducing the powers in their moment assumptions.

More recently, some authors have considered working with infinite dim\-ensional targets, part\-icularly in the case that the target has a density with respect to the law of a Gaussian process. This includes \citep{mattingly2012diffusion} which covers the RWM case and \citep{pillai2012optimal} which covers the MALA case. These papers allow for a non-trivial dependence structure, but only under the strong assumption of absolute continuity with respect to an infinite-dimensional Gaussian distribution. They show that the $\approx 23.4\%$ and $\approx 57.4\%$ optimal acceptance rates for RWM and MALA respectively carry over to infinite dimensional distributions which have densities with respect to the laws of a Gaussian processes. Though these papers allow for a non-trivial dependence structure, they do not consider the optimal shaping problem.

Lastly, \citep{zanella2017dirichlet} utilise the theory of Dirichlet forms to establish weak converg\-ence of the infinite dimensional limit process for targets of the same form as in \citep{roberts1997weak}. 
Using the powerful theory of Mosco convergence, they are able to eliminate many of the assum\-ptions \citep{roberts1997weak}. 
In particular, that paper requires no additional smooth\-ness or moment assumptions.
In fact, they are able to demonstrate convergence of the Markov semigroup with assumptions on $D\log\pi$ which are weaker than Lipschitz continuity, though to ensure weak convergence of the path processes they do require Lipschitz continuity.
Hence, our present paper's assumptions used to demonstrate weak convergence of the infinite dimen\-sional paths are the same as in that work.

\subsection{Notation and Definitions}
\label{subsec:notation}
Let $\pi$ be the Lebesgue density of a prob\-ability distribution on $\RR^k$.

For each $d\in\NN$, let $\Pi_d = \pi^{\otimes d}$. Then $\Pi_d$ is the joint density for $d$ independent blocks, each of dimension $k$, identically distributed according to $\pi$. 

The ``accelerated, continuous time'' Random Walk Metropolis (RWM) proc\-ess with stationary distribution $\Pi_d$, and mean-0 proposal distribution: 
\[
	\Nn_d(l^2\Lambda):=\Nn\lcr({0,\frac{l^2}{(d-1)} I_d\otimes \Lambda }) \ , 
\] 
for $\Lambda\in\RR^{k\times k}$ symmetric positive definite, is the Markov proc\-ess, $\vX_d(t)$ such that $\vX_d(0)\sim \Pi_d$ with infinitesimal generator $\hat G^{l,\Lambda}_d$ given by:
\[
[\hat G^{l,\Lambda}_d f](x) = kd \ \EEE{Z\sim\Nn_d(l^2\Lambda)} \lcr[{(f(x+Z) - f(x))\lcr({1\wedge\frac{\Pi_d(x+Z)}{\Pi_d(x)}})}]\ 
\]
for $f\in \ol C(\RR^{kd})$. 

Note that $I_d \otimes \Lambda$ is the $kd \times kd$ block diagonal matrix with $d$ blocks of size $k\times k$ all equal to $\Lambda$:
\[
	I_d\otimes\Lambda = \lcr[{\begin{matrix} 
		\Lambda & 0 		& 0 		& \cdots & 0 		& 0 		\\
		0 		& \Lambda	& 0 		& \cdots & 0 		& 0 		\\
		0 		& 0			& \Lambda	& \cdots & 0 		& 0 		\\		
		\vdots 	& \vdots	& \vdots	& \ddots & \vdots 	& \vdots 	\\
		0 		& 0			& 0			& \cdots & \Lambda 	& 0		 	\\		
		0 		& 0			& 0			& \cdots & 0		& \Lambda 	\\						
		\end{matrix}}]
\]
Equivalently, $\vX_d(t)$ is the pure jump Markov process with jumps occurring at exponentially distributed intervals with rate $kd$, and jumps distributed according to the RWM transition kernel with proposal distribution $\Nn_d(l^2\Lambda)$, started according to the stationary distribution $\Pi_d$. Note that, since the RWM transition kernel has a non-zero probability of remaining at the same point, the continuous version may ``jump to the same point'' (this occurs if the metropolis accept/reject step rejects a proposal).

Let $\vX^{(i)}_d(t)$ be the stochastic process on $\RR^k$ consisting of the $i$th $k$-dimen\-sional block of $\vX_d(t)$. In general, this process is not Markov. For $i<j$, let $\vX^{(i):(j)}_d(t)$ be the stochastic process on $\RR^k$ consisting of the $i$th, $(i+1)$th,...,$j$th $k$-dimensional blocks of $\vX_d(t)$, so that $\vX^{(i):(j)}_d(t)$ has paths which take values in $\RR^{k(j-i)}$.

For each $r\in\NN$, the anisotropic Langevin diffusion with stationary distrib\-ution $\pi^{\otimes r}$, and anisotropy matrix $\Lambda$, and time-scaling factor $l^2 a_\Lambda (l)$,  $\vX^r(t)$, is the Markov process with $\vX^r(0)\sim\pi^{\otimes r}$ and infinitesimal generator $G^{l,\Lambda}$ given by:
\[ \label{def:gen-diff}
	[G^{l,\Lambda} f]
	 	& = k l^2 a_\Lambda(l) \lcr({\frac{1}{2} [I_r\otimes\Lambda] : (\grad^2 f) + \frac{1}{2}[\grad\log\pi^{\otimes r}]'[I_r\otimes\Lambda] (\grad f)})
\]
for a sufficiently large class of functions $f$, and where
\[
	a_\Lambda(l) 
		& = 2\Phi \lcr({ - \frac{l \sqrt{\Sigma:\Lambda}}{2}}) \\
	\Sigma 
		& = \VVar{X\sim\pi}(\grad\log\pi(X))\\
	A:B			
		& = \sum_{i,j} A_{ji} B_{ij}\ .
\]
Equivalently, it is the diffusion process (with initial distribution $\pi^{\otimes k}$) satisfy\-ing the SDE:
\[
d \vX^r(t) = k l^2 a_\Lambda(l) [I_r\otimes\Lambda] [\grad\log\pi^{\otimes r}(\vX^r(t))]\ dt + \sqrt{2 k l^2 a(l)} I_r \otimes \sqrt{\Lambda}\ d\vB(t)  
\]
where $\vB(t)$ is a standard $kr$-dimensional Wiener process, and $\sqrt{\Lambda}$ is the symmetric positive definite square-root of the symmetric positive definite matrix $\Lambda$. Thus, $\vX^r$ is the same process as $r$ independent copies of the $\vX^1$ appended together.

Later, in the case of $r=1$, we will also compare this to the generator of a similar diffusion, with the same stationary measure and anisotropy matrix, at standardized speed:
\[
	[G^\Lambda f] (x)
	 	& =k \lcr({\frac{\Lambda}{\Lambda:\Sigma} : (\grad^2 f) + [\grad\log\pi]'\frac{\Lambda}{\Lambda:\Sigma} (\grad f)})
\]
The choice of time-scaling used for the standardized speed, as will be shown in Corollary \ref{cor:opt-scale}, corresponds to $G^{l,\Lambda}$ for the optimal choice of $l$ given $\pi$ and $\Lambda$ up to universal constants (not dependent on $k$, $\pi$, or $\Lambda$).

We make the following key assumption about $\pi$ throughout this work:
\begin{assumption*}[\textbf{A1}]\label{ass:grad-log-lip} $\grad\log\pi$ is $L$-Lipschitz continuous for some $L>0$. 
\end{assumption*}
Some geometric and analytic consequences of this assumption are that:
\begin{proposition}[Summary of consequences of assumption (A1) on $\pi$]\label{prop:summary-of-gll}
The assumption that $\grad\log\pi$ is Lipschitz continuous implies all of the foll\-owing:
\begin{enumerate}[label=(\alph*)]
\item $\grad\log\pi$ is differentiable (Lebesgue-)almost everywhere, and \[\norm{\grad^2\log\pi} \leq L\] where it exist (by Rademacher's theorem, see \citep{federer1969geometric}).
\item The tails of $\pi$ are at least as heavy of those of a Gaussian distribution. In fact it can be bounded below by a tangent Gaussian curve with variance-covariance matrix $\frac{1}{L} I$ at each point. (Lemma \ref{lem:tan-min}). This fur\-ther implies that $\supp(\pi) = \RR^k$.
\item $\pi$ is uniformly bounded above (Lemma \ref{lem:bound-abv}).
\item $\pi$ is Lipschitz (Lemma \ref{lem:lip-gll}).
\item $\pi$ has a broadly applicable integration by parts formula (Corollary \ref{cor:ibps-gll}): For any $f: \RR^k\to \RR$ which is locally Lipschitz, with $\grad f(X)$ and $f(x)\grad\log\pi(x)$ integrable (w.r.t. $\pi(x)dx$) we have
\[
    \EEE{X\sim\pi} f(X)\grad\log\pi(X)  = -\EEE{X\sim\pi}\grad f(X) \ .	
\]
Similar formulas also hold for Jacobians of locally Lipschitz functions $f:\RR^k\to\RR^m$, and for divergences of locally Lipschitz functions $f:\RR^k\to\RR^k$.
\item The following identities hold (Lemma \ref{lem:gll-glp-var}): 
\[
	\EEE{X\sim\pi}[\grad\log\pi(X)] 
		& = 0 \\ 
	\VVar{X\sim\pi}[\grad\log\pi(X)] 
		&= - \EEE{X\sim\pi}[\grad^2 \log\pi(X)] \ .
\]
\item If $X\sim\pi$ then $\grad\log\pi(X)$ is sub-Gaussian with proxy-variance $L$ (Lemma \ref{lem:gll-subg}). Hence all polynomial moments of $\grad\log\pi$ and $\grad^2\log\pi$ are finite (Remark \ref{rem:gll-moments}).
\end{enumerate}
\end{proposition}

\newpage

\section{Main Results}\label{sec:keyresults}
%!TEX root = ../blockiid_main.tex
\setlength{\parindent}{0pt}
\setlength{\parskip}{6pt}
\subsection{Weak Convergence in the Skorohod Topology}
\label{subsec:weak}
\begin{theorem}[Weak convergence of finite dimensional processes in the Skorohod topology on $\RR^{rk}$]\label{thm:weak-conv}
Under the definitions above, if assumption \hyperref[ass:grad-log-lip]{(A1)} holds then (for each $r\in\NN$) $\vX_d^{(1):(r)}$ converges weakly in the Skorohod topology to $\vX^r$ as $d\to\infty$.
\end{theorem}

The proof of this result is the content of Section \ref{sec:pf-weak-conv}.

By bootstrapping our result on weak convergence of the stochastic process of finite dimensional blocks we are also able to demonstrate weak convergence of the infinite dimensional process. 

Let $\vY_d(t) = (\vX_d(t),0,0,...)$, so that $\vY(t)\in\RR^\infty$ (this is similar to the processes considered in \citep{zanella2017dirichlet}).

By the Kolmogorov extension theorem (see for example \citep{tao2011introduction}, section 2.4 therein) applied to the sequence $\vX^r$ over $r\in\NN$, there is a unique (in probability law) process $\vX^\infty$ taking values in $\RR^\infty$ such that the marginal process of the first $kr$ components has the same distribution as $\vX^r$. To prove convergence of $\vY_d$ to $\vX^\infty$ in the Skorohod topology (of $\RR^\infty$ with the product topology) we need the following lemma:

\begin{lemma}\label{lem:sep-pt-prod} If $\RR^\infty$ is equipped with the metric 
\[\label{eqn:prod-metric}
	r(x,y) 
	 &= \sum_{i\geq 1} 2^{-i} (\abs{x_i- y_i} \wedge 1)
\] 
(which generates the product topology) then 
\[
	M 
		& = \set{f\circ \rho_j \st j\in\NN, f\in \ol C(\RR^j)}
\] 
strongly separates points (where $\rho_j :\RR^\infty\to\RR^j$ is the projection map onto the first $j$ components).
\end{lemma}
\begin{proof}
Fix $1>\delta>0$ and $x\in \RR^\infty$, and let $m_\delta = \ceil{-\log_2(\delta)}$. Let 
\[
	h_{x,\delta}(z) = \frac{2}{\delta} \lcr({\frac{\delta}{2} - \sum_{i=1}^{m_{\delta}+1} 2^{-i} (\abs{x_i- z_i} \wedge 1)})_+ \ .
\]
Notice that $h_{x,\delta} \in M$. Obviously $h_{x,\delta}(x) =1$.

Suppose $y\in\RR^\infty$ such that $r(x,y) \geq\delta$; since 
\[
	\sum_{i=m_{\delta}+2}^{\infty} 2^{-i} (\abs{x_i- y_i} \wedge 1) \leq 2^{-(m_\delta+1)} \leq \delta / 2 \ ,
\]
then
\[
	\sum_{i=1}^{m_{\delta}+1} 2^{-i} (\abs{x_i- y_i} \wedge 1) \geq \delta/2 \,
\] 
and hence $h_{x,\delta}(y) =0$. 
\end{proof}

\begin{theorem}[Weak convergence  of the infinite dimensional process in the Skorohod topology of $\RR^\infty$]\label{thm:weak-conv-sk-prod}
Under the definitions above, if assumption \hyperref[ass:grad-log-lip]{(A1)} holds then $\vY_d$ converges weakly in the Skorohod topology (of $\RR^\infty$ with the metric $r$ defined in Equation (\ref{eqn:prod-metric})) to $\vX^\infty$. 
\end{theorem}
\begin{proof}
From Lemma \ref{lem:sep-pt-prod}, $M$ (as defined above) strongly separates points. Consider any finite subset of $M$, say $\set{h_1,...,h_n}$. Then without loss of generality there exists an $m\in\NN$ with and a set of functions $\set{f_1,...,f_n}\subset \ol{C}(\RR^{mk})$ with $h_i = f_i\circ \rho_{mk}$ for all $i\in\set{1,...,n}$. 

If $E$ is a metric space, and $f:E\to \RR$ then define its ``lift'' onto $D_E[0,\infty)$ as $\tilde f: (t \mapsto X(t)) \mapsto (t\mapsto f(X(t))$, so that $\tilde f:D_E[0,\infty)\to D_\RR[0,\infty)$. If $f$ is continuous in the topology on $E$ then $\tilde f$ must be continuous in the Skorohod topology on $D_E[0,\infty)$. This fact is proven in \citep{jakubowski1986skorokhod} (theorem 4.3 therein)\footnote{In fact, \citep{jakubowski1986skorokhod} tells us the stronger result,  that the Skorohod topology on $D_E(0,\infty]$ is the coarsest topology for which the lifts of all continuous functions are continuous.}.

Now, since all of the finite dimensional processes of $\vY_d$ converge weakly in the Skorohod topology, and since the lift of a continuous function on $\RR^{km}$ to $D_{\RR^{km}}[0,\infty)$ is continuous then, by the continuous mapping theorem, 
\[
(h_1,...,h_n)(\vY_d) \disttosk (h_1,...,h_n)(\vX^\infty) \ .
\] 
By \citep{ethier2009markov} (corollary 9.2 therein) this is sufficient to ensure that $\vY_d$ converges weakly in the Skorohod topology to $\vX^\infty$.
\end{proof}
 
\subsection{Optimal Scaling (with fixed shaping)}
\label{subsec:scaling}
For the rest of Section \ref{sec:keyresults}, for simplicity, we assume that $r=1$ so that we consider only the limiting dynamics of a single block. However, the results do carry over to mult\-iple blocks, and even to the infinite dimensional limit, because of tensor\-isation properties of spectral gaps \citep{bakry2013analysis} and covariances.

Having shown that the limiting process is a Langevin diffusion, it is natural to try to select the tuning parameters, $(l,\Lambda)$ such that the limiting diffusion mixes as fast as possible. For a fixed choice of $\Lambda$, if we change $l$ then we only change the time scaling of the process. That is, for different values of $l$, we are running a diffusion with the dynamics given by $G^\Lambda$ accelerated by a factor of $l^2 a_\Lambda(l) (\Lambda :\Sigma) /2$. Thus we find that the optimal choice of $l$ for a fixed $\Lambda$ is easy to determine; we need only maximise the time-change factor $l^2 a_\Lambda(l)$ in order to make the diffusion move towards stationarity as quickly as possible. As in \citep{roberts1997weak} we will characterise the optimal scaling both in terms of the value of the scaling factor, $l$, and in terms of the limiting average acceptance probability for the RWM algorithm. The optimal choice of $\Lambda$ will prove more challenging to derive as changing $\Lambda$ does not induce only a time-change on the dynamics of the process.

\begin{lemma}(Limiting Acceptance Rate)\label{lem:lim-acc}
	The limiting acceptance rate for the RWM proposals of $\vX_d$ is $a_\Lambda(l)$. That  is to say:
	\[
		\lim_{d\to\infty}\hspace{-1.5em} \EEE{\subalign{\qquad X&\sim\Pi_d \\ Z&\sim\Nn_d(l^2\Lambda)}}\lcr[{1 \wedge \frac{\Pi_d(X+Z)}{\Pi_d(X)}}] = a_\Lambda(l) = 2\Phi \lcr({ - \frac{l \sqrt{\Sigma:\Lambda}}{2}})
	\]
	where $\Sigma = \VVar{X\sim\pi}[\grad\log\pi(X)]$ and $(:)$ is the Frobenius inner product.
\end{lemma} 
This result is proved as a step in the proof of Theorem \ref{thm:weak-conv}, in Section \ref{sec:pf-weak-conv}. 

\begin{corollary}[Optimal Scaling of $l$ for fixed $\Lambda$]\label{cor:opt-scale}
The optimal scaling over $l$ (in terms of the fastest time-change) for a fixed $\Lambda$ is $l_\Lambda \approx \frac{2.38}{\sqrt{\Sigma : \Lambda}}$. This is the $l_\Lambda$ which solves $a(l_\Lambda) \approx 0.234$, as in the original \citep{roberts1997weak} result. The limiting diffusion corresponds to $G^\Lambda$ sped up (or slowed down) by a factor of is $l_\Lambda^2 a(l_\Lambda) (\Lambda :\Sigma) /2 \approx 0.66$. This acceleration factor is universal; it does not depend on $k$, $\pi$, or $\Lambda$.
\end{corollary}

\begin{proof}
For fixed $\Lambda$, $G^{l,\Lambda} = l^2 a_\Lambda(l) (\Lambda :\Sigma) G^\Lambda / 2$. Thus $G^{l,\Lambda}$ corresponds to $G^\Lambda$ accelerated (decelerated) by a factor of $l^2 a_\Lambda(l) (\Lambda :\Sigma) /2$. Hence, to maximize the speed of $G^\Lambda$ over the choice of $l$ we need only maximize $h(l): =l^2 a_\Lambda(l) = 2 l^2 \Phi \lcr({ - \frac{l \sqrt{\Sigma:\Lambda}}{2}})$ over $l$.

This is equiavlent to the original optimisation from \citep{roberts1997weak}. Notice that:
\[
h(l) = \frac{8}{\Sigma:\Lambda} \lcr({\frac{l\sqrt{\Sigma:\Lambda}}{2}})^2\Phi \lcr({ - \frac{l \sqrt{\Sigma:\Lambda}}{2}})
\]
Taking $\omega = \frac{l\sqrt{\Sigma:\Lambda}}{2}$ we can maximise instead:
\[
\tilde h (\omega) = \omega^2 \Phi(-\omega)
\]
This may be done numerically to get $\omega_\star \approx 1.1906$, $\tilde h(\omega_\star) \approx 0.165717$. Then solving for $l$ yields $l_\Lambda = \frac{2 \omega_\star}{\sqrt{\Sigma:\Lambda}} = \frac{\approx 2.3812}{\sqrt{\Sigma:\Lambda}}$ and $h(l_\Lambda) = \frac{8\tilde h(\omega_\star)}{\Sigma:\Lambda} = \frac{\approx 1.32574}{\Sigma:\Lambda}$.

Hence $G^{l_\Lambda,\Lambda} = 4 \tilde h(\omega_\star) G^{\Lambda} = (\approx 0.66) G^{\Lambda}$
\end{proof}

\subsection{Optimal Shaping I: Optimal Spectral Gaps in Special Cases}
\label{subsec:shaping1}
For the rest of this section, we work only with $G^\Lambda$. This is equivalent to  assuming that the optimal scaling is always used: $l = l^\Lambda$, and the universal acceleration factor $4 \tilde h(\omega_\star)\approx 0.66$ is ignored.

We note that, since $\vX$ is Feller (Lemma \ref{lem:verify-ek-1-4}), the spectrum of $G^\Lambda$ is a subset of the non-positive real line. Also, there is an eigenvalue at $0$ corresponding to the constant function. If $G^\Lambda$ has a spectral gap for each $\Lambda$, then the ideal choice for tuning parameters would be those that maximise the spectral gap of $G^\Lambda$. This would in turn optimise the rate of convergence to stationarity of the diffusion process (see, for example \citep{bakry2013analysis}). 

We add the following assumption when needed, in order to ensure that the optimisation over $\Lambda$ is meaningful.
\begin{assumption*}[A2]\label{ass:spec-gap} For all $\Lambda$, $G^\Lambda$ has a spectral gap. More precisely, for all symmetric positive definite $\Lambda\in\RR^{k\times k}$ there exists $\rho_\Lambda >0$ such that for all $f\in\Dd(G^\Lambda)$:
\[
	\VVar{X\sim\pi} f(X) 
		& \leq -\frac{1}{\rho_\Lambda} \EEE{X\sim\pi}\lcr({ f(X)\ [G^\Lambda f](X)}) \\
		& = \frac{1}{\rho_\Lambda} \EEE{X\sim\pi}\lcr({ \grad f(X)' \frac{\Lambda}{\Lambda:\Sigma} \grad f(X)}) \ .
\]
\end{assumption*}
The spectral gap of $G^\Lambda$ is the largest possible $\rho_\Lambda$.

\begin{remark}[The spectral gap assumption is satisfiable]
The curvature-dimension condition of \citep{bakry2013analysis} provides one way to verify assumption \hyperref[ass:spec-gap]{(A2)}. A simple example is that if $\log\pi$ is strongly concave, then the condition is satisfied with $\frac{1}{\rho_\Lambda} = \essinf_{x\in\RR^k} \lambda_{1}(-\grad^2\log\pi(x))>0$ for all $\Lambda$ (where $\lambda_1$ is the function which returns the minimal eigenvalue of a matrix. In this case, $\frac{1}{\rho_\Lambda}$ is the strong convexity parameter). 
\end{remark}

As mentioned before, the optimal shaping problem turns out to be a much more difficult than the optimal scaling was. We solve this problem exactly, first when the target is a multivariate normal distribution, and second when the target density is a rotated independent product of a scale family. 

For more general target distributions, the problem of optimising the spectral gap is not so easily approachable as it is not known (at least by the author of this work) in general how to directly compute the spectral gap of the generator for a Langevin diffusion process (or even to determine sharp conditions for when there is a spectral gap at all) or how the spectrum transforms under a change in anisotropy. Instead, we optimise a surrogate measure of the process' speed: the rate of decay of autocorrelations of functions of $\vX$ near lag-$0$. As will be discussed in the next section, this is a continuous time analogue of a common heuristic for the short term mixing properties of MCMC algorithms in discrete time, as well as a relaxation of the spectral gap optimisation problem which we would strive to solve if we could. This surrogate will also give novel `speed limits' on the convergence of RWM -- upper bounds on the spectral gap of the generator which limit the convergence rate in terms of properties of $\pi$.

\begin{theorem}[Optimal Shaping when $\pi \equiv\Nn(\mu,\Gamma)$] 
\label{thm-opt-shape-1}
When $\pi$ is the density of a $\Nn(\mu,\Gamma)$ distribution, then $\Sigma:= \VVar{X\sim\pi}(\grad\log\pi(X)) = \Gamma^{-1}$ and the spectral gap of $G^\Lambda$ is maximised by taking the shaping matrix to  be (proportional to) the covariance of the target distribution; $\Lambda = \Sigma^{-1} = \Gamma$. The convergence to stationarity is $\frac{\tr(\Sigma)}{ k \lambda_1(\Sigma)}$ times faster when using the optimal shaping as opposed to spherical shaping, where $\lambda_1(\Sigma)$ is the minimal eigenvalue of $\Sigma$.
\end{theorem}
\begin{proof}
Let $\spec(A)$ denote the spectrum of the operator $A$.

Without loss of generality, $\mu = 0$. In this case, $\grad\log\pi(x) = -\Gamma^{-1} x$, so 
\[\Sigma = \EEE{X\sim\pi}(\Gamma^{-1} X X'\Gamma^{-1}) = \Gamma^{-1}\ .\]
Now, we note that:
\[
  [G^\Lambda f](x) 
    & = \frac{k}{\Lambda:\Sigma} \lcr({\Lambda : \grad^2 f(x) + (-x' \Gamma^{-1}) \Lambda \grad f(x)}) \\
    & =\frac{k}{\Lambda:\Sigma} \lcr({\Lambda : \grad^2 f(x) + x' (-\Sigma\Lambda)\grad f(x)})
\]
From \citep{metafune2002spectrum} we know that 
\[
	\spec(G^\Lambda) = \set{\sum_{s\in\spec(B)} s n_s \st\  n_s \in\NNO\ \forall s\in \spec(B)}\ , 
\] where $B = \frac{-k\Sigma\Lambda}{\Lambda:\Sigma}$. Therefore the spectral gap of $G^\Lambda$ is exactly the smallest eigenvalue of $\frac{k\Sigma\Lambda}{\Lambda:\Sigma}$. Now, letting $A = \Sigma\Lambda$, and letting $\set{\lambda_i(A)}_{i=1}^k$ be the (non-decreasing) eigenvalues of $A$ we can solve: 
\[
  \argmax_{A} \frac{\lambda_1(A)}{\sum_{i=1}^k \lambda_i(A)}
\]
This function is bounded above by $1/k$ since $a_1\leq a_i$ for all $1\leq i \leq k$ and the function is equal to $1/k$  if and only if $A = \gamma I$ for some $\gamma\neq 0$. Hence the optimal spectral gap is achieved at $\Sigma \Lambda = I$. Therefore $\Lambda^\star = \Sigma^{-1} =\Gamma$. 

We also find that the spectral gap of $G^{\Lambda^\star}$ is $1$. On the other hand, the spectral gap of $G^{I}$ is $\frac{k\lambda_1(\Sigma)}{\tr(\Sigma)} = \frac{\lambda_1(\Sigma)}{\ol \lambda(\Sigma)}$, where $\ol \lambda(\Sigma)$ is the average eigenvalue of $\Sigma$. Therefore, the convergence is $\frac{\ol \lambda(\Sigma)}{\lambda_1(\Sigma)}$ times faster when using the optimal shaping as opposed to spherical shaping.
\end{proof}

\begin{theorem}[Optimal Shaping when $\pi$ is a rotated independent product of a scale family]\label{thm-opt-shape-2}
Suppose that $\pi_1$ is a probability density on $\RR$, with $D[\log\pi_1]$ Lipschitz, and the one-dimensional generator 
\[
	G_1 f = \lcr({D^2[f] + D[\log\pi_1] D[f]})
\] 
satisfies assumption \hyperref[ass:spec-gap]{(A2)} with spectral gap $\rho = \lambda^{\star}$.

Let $c_i>0$ for each $1\leq i \leq k$ and let $Q$ be a unitary transformation, so $Q' = Q^{-1}$. 

Let $\pi(x) = \prod_{i=1}^k c_i \pi_1(c_i e_i' Q x)$, so that $X\sim \pi$ if and only if $c_i e_i'Q X \mysim{iid} \pi_1$ for $1\leq i \leq k$. (Where $e_i$ are the standard basis vectors).
 
Then 
\[
	\Sigma := \VVar{X\sim\pi}(\grad\log\pi(X)) = \sigma^2 Q' \diag(c_i^2) Q
\] 
where $\sigma^2 = \EEE{X_1\sim\pi_1}\lcr[{D\log\pi_1(X_1)^2}]$ and the spectral gap of $G^\Lambda$ is maximised (among those $\Lambda$ of the form $Q' D Q$ with $D$ diagonal) by $\Lambda = \Sigma^{-1}$. The convergence to stationarity is $\frac{\tr(\Sigma)}{ k \lambda_1(\Sigma)}$ times faster when using the optimal shaping as opposed to spherical shaping, where $\lambda_1(\Sigma)$ is the minimal eigen\-value of $\Sigma$.
\end{theorem}
\begin{proof}
We first compute 
\[
    \grad\log\pi(x) 
        & = \sum_{i=1}^k c_i Q' e_i [D\log\pi_1](c_i e_i' Q x) \,
\]
and
\[
   &\hspace{-2em} \EEE{X\sim\pi}  \grad\log\pi(x) \grad\log\pi(x)' \\
        &= \sum_{i=1}^k \sum_{j=1}^k c_i c_j Q' e_i e_j' Q \EEE{X\sim\pi}  [D\log\pi_1](c_i e_i' Q X) [D\log\pi_1](c_j e_j' Q X) \\
        &= \sum_{i=1}^k \sum_{j=1}^k c_i c_j Q' e_i e_j' Q \EEE{Y\sim\pi_1^{\otimes k}}  [D\log\pi_1](Y_i) [D\log\pi_1](Y_j) \\
        &= \sum_{i=1}^k \sum_{j=1}^k c_i c_j Q' e_i e_j' Q \delta_i^j \sigma^2\\
        &= \sigma^2 Q' \diag(c_i^2) Q                   
\]
Thus $\Sigma = \sigma^2 Q' \diag(c_i^2) Q$.

Suppose that $Q$ = $I$. Consider the generator $G_1$ as in the theorem statement. $G_1$ generates a Feller semigroup with stationary measure $\pi_1$, and so there is a complete basis (for $L_2(\RR,\pi_1)$) of eigenfunctions for $H$ (see section 4.7 of \citep{pavliotis2014stochastic}). Let these eigenfunctions be $\set{\phi_\alpha}_{\alpha\in\NNO}$ and the corresponding eigenvalues be $\set{\lambda_\alpha}_{\alpha\in\NNO}$. 

Under the assumption that a spectral gap exists for $G_1$ with tight constant $\lambda^\star$, we may assume that $\lambda_0 = 0$, $\lambda_1 = \lambda^\star$ and $\lambda_\alpha\geq \lambda^\star$ for $\alpha\geq 2$. 

Then $S = \set{\prod_{i=1}^k \phi_{\alpha_i}\circ C_i \st \ \alpha_i \in\NN\ \forall i}$ is a complete basis for $L_2(\RR^k, \pi)$, where $C_i:x\mapsto c_i x_i$. Moreover, $S$ contains only eigenvectors for $G^\Lambda$: 
\[
    & G^\Lambda \lcr[{\prod_{i=1}^k \phi_{\alpha_i}\circ C_i}]\\
      & =  \frac{k}{\Lambda:\Sigma} \sum_{i} \Lambda_{ii} c_i^2 \lcr[{(D^2\phi_{\alpha_i}\circ C_i + (D\log\pi_1)(D[\phi_{\alpha_i}\circ C_i])) \prod_{j\neq i} \phi_{\alpha_j}\circ C_j}]\\
      & = \frac{k}{\Lambda:\Sigma} \sum_{i} \Lambda_{ii} c_i^2 \lambda_{\alpha_i} \lcr[{\prod_{i=1}^k \phi_{\alpha_i}\circ C_i}]
\]
Therefore, $\spec(G^\Lambda) = \set{k\frac{\sum_{i=1}^k \Lambda_{ii} c_i^2 \lambda_{\alpha_i}}{\sum_{i=1}^k \Lambda_{ii} c_i^2} \st\  \alpha_i \in\NNO \  \forall 1\leq i\leq k}$. Hence the spectral gap for $G^\Lambda$ is the minimal eigenvalue of $\frac{k\Lambda\Sigma}{\Lambda:\Sigma}$. The optimal spectral gap is thus achieved, as in Theorem \ref{thm-opt-shape-1}, by $\Lambda = \Sigma^{-1}$.

For general $Q$, unitary, let $M_Q f = f\circ Q$. Then $G^\Lambda$ has the same spectrum as $G^\Lambda_Q = M_Q^{-1} G^\Lambda M_Q$ since these are similar operators.

\[
	[G^\Lambda_Q f] (x) 
		& = \frac{k\Lambda}{\Lambda : \Sigma} : \lcr({Q'\grad^2 f(x)Q + \grad\log\pi(Q'x) \grad f (x)' Q }) \\
		& = \frac{k\Lambda}{\Lambda : \Sigma} : \lcr({Q'\grad^2 f(x)Q + Q' \sum_{i=1}^k c_i  [D\log\pi_1](c_i x_i) e_i \grad f (x)' Q }) \\
		& = \frac{kQ\Lambda Q'}{\Lambda : \Sigma} : \lcr({\grad^2 f(x) + \sum_{i=1}^k c_i  [D\log\pi_1](c_i x_i) e_i \grad f (x)' })		
\]
This generator is of the same form as $G^{Q\Lambda Q'}$ when $Q=I$, except with $\Sigma$ replaced by $Q\Sigma Q'$. Thus the spectrum of this operator is optimised when $Q\Lambda Q' = (Q\Sigma Q')^{-1}$, which occurs exactly when $\Lambda = \Sigma^{-1}$.
\end{proof}

\subsection{Optimal Shaping II: Decay of Auto\-correlations and Speed Limits}
\label{subsec:shaping2}
In this section we first describe how the generator $G^\Lambda$ is related to the slope of the autocorrelation of functions of $\vX$. Then we consider a relaxation of the spectral gap problem to searching for smaller subspaces of $\Dd(G^\Lambda)$.

\begin{lemma}[Relationship between autocorrelation and generators]\label{lem:corr-gen}
For any $f\in \Dd(G^\Lambda)$ which is not is not almost everywhere constant. Let $Y(t) = f(\vX(t))$. Then 
\[
    \frac{d}{dt} \Cov(Y(0),Y(t)) \big\given_{t=0^+} 
        & = \EEE{X\sim\pi}\lcr[{f(X)\ [G^\Lambda f](X)}] \\
        &  = -\EEE{X\sim\pi}\lcr[{\grad f(X) ' \frac{k\Lambda}{\Lambda:\Sigma} \grad f(X)}]\]
and
\[
    \frac{d}{dt} \Corr(Y(0),Y(t)) \big\given_{t=0^+} 
        & = \frac{\EEE{X\sim\pi}\lcr[{f(X)\ [G^\Lambda f](X)}]}{\VVar{X\sim\pi} [f(X)]} \\
        &  = -\frac{\EEE{X\sim\pi}\lcr[{\grad f(X) ' \frac{k\Lambda}{\Lambda:\Sigma} \grad f(X)}]}{\VVar{X\sim\pi} [f(X)]}
\]
Hence, the spectral gap of $G^\Lambda$ is given by 
\[
	\lambda_{\star}^\Lambda 
		& = \inf_{f\in C^2(\RR^k)\cap\Dd(G^\Lambda)}\abs{\frac{d}{dt} \Corr(Y(0),Y(t)) \big\given_{t=0^+}} \\
		& = \inf_{f\in C^2(\RR^k)\cap\Dd(G^\Lambda)}\frac{k\EEE{X\sim\pi}\lcr[{\grad f(X) ' \Lambda \grad f(X)}]}{\VVar{X\sim\pi} [f(X)]}
\]
\end{lemma} 
\begin{proof}
From It\^o's lemma, for $t>0$:
\[
    Y(t)-Y(0) 
    	& = \int_0^t G^\Lambda f(\vX(t)) dt + \int_0^t \grad f(\vX(t))' \frac{k\Lambda}{\Lambda:\Sigma} \grad\log\pi(\vX(t)) d\vB(t)\\
    	& = \int_0^t G^\Lambda f(\vX(t)) dt + M_t
\]
In this expansion, $M_t$ is a $\set{\vB(t),\vX(0)}$-martingale starting at $0$, and hence is uncorrelated with Y(0) which is $\sigma(\vX(0))$-measurable. Moreover, by Fubini's theorem, 
\[
	\EE\lcr[{ \int_0^t G^\Lambda f(\vX(t)) dt}]
		& = \int_0^t \EE [G^\Lambda f(\vX(t))] dt
		 = 0 \,
\]
where the last equality follows by integration by parts. Hence, using Fubini's theorem again:
\[
	\Cov(Y(0),Y(t))
		& = \EE \lcr[{Y(0) \lcr({Y(0) + \int_0^t G^\Lambda f(\vX(t)) dt + M_t})}]\\
		& = \Cov(Y(0),Y(0)) + \int_0^t \EE [f(\vX(0) G^\Lambda f(\vX(t))] dt		
\]
Now, using the fundamental theorem of calculus, 
\[
    \frac{d}{dt} \Cov(Y(0),Y(t)) \bigg\given_{t=0^+} 
        & = \EEE{X\sim\pi}\lcr[{f(X)\ [G^\Lambda f](X)}]
\]
Applying integration by parts, we get:
\[
    \frac{d}{dt} \Cov(Y(0),Y(t)) \bigg\given_{t=0^+} 
        & = \EEE{X\sim\pi}\lcr[{f(X)\ [G^\Lambda f](X)}] \\
        &  = -\EEE{X\sim\pi}\lcr[{\grad f(X) ' \frac{k\Lambda}{\Lambda:\Sigma} \grad f(X)}]
\]
Finally, the statement regarding correlations follows from the definition of correlation in terms of covariance.
\end{proof}
Since $C^2(\RR^k)\cap\Dd(G^\Lambda)$ is dense in $\Dd(G^\Lambda)$, if assumption \hyperref[ass:spec-gap]{(A2)} holds, then the spectral gap of $G^\Lambda$ is exactly the worst-case (negative) lag-$0$ slope of the autocorrelation of functions in $C^2(\RR^k)\cap\Dd(G^\Lambda)$.

Thus, for convenience of solution, one may consider in place of $C^2(\RR^k)\cap\Dd(G^\Lambda)$ a smaller class of functions, $\Ff\subsetneq C^2(\RR^k)\cap\Dd(G^\Lambda)$ over which to solve  
\[\label{eqn:opt-prob}
\max_{\Lambda \succ 0}\min_{f\in \Ff}\frac{\EEE{X\sim\pi}\lcr[{\grad f(X) ' \frac{k\Lambda}{\Lambda:\Sigma} \grad f(X)}]}{\VVar{X\sim\pi} [f(X)]}
\]
The solution to this relaxed problem may be interpreted as optimising the worst case auto\-correlation (in a neighbourhood of lag-$0$) among functions from $\Ff$. When $\Ff$ is itself a subspace of $C^2(\RR^k)\cap\Dd(G^\Lambda)$, then the solution may be interpreted as optimising the spectral gap of the restricted generator $G^\Lambda \big\vert_\Ff : \Ff \to G^\Lambda(\Ff)$ which is also a linear operator (which may be bounded or unbounded, depending on the choice of $\Ff$). Obviously, the solution to the restricted problem, for some choice of $\Ff$, is in fact optimising an upper bound on the spectral gap of $G^\Lambda$, not the actual spectral gap of $G^\Lambda$.

Suppose that $X\sim\pi$ has finite second moments, so that (for each $v\in\RR^k$) if the process $v'\vX(t)$ is started from the stationary distribution, it admits a stationary auto\-correlation function. A heuristic commonly used to tune MCMC algorithms in discrete time is the lag-1 auto\-correlation of each component ($X_i$), with smaller lag-$1$ auto\-correlation being better. The continuous time analogue of minimising the discrete-time lag-$1$ auto\-correlation is maximising the (absolute value of the) slope of the auto\-correlation function at lag $0$. Rather than considering only each component projection, $\Ff_0 =\set{x\to e_i'x: i\in\set{1,...,k}}$, we will consider their span, the subspace $\Ff = \set{x\to v'x \ :\ v\in\RR^k\setminus\set{0} }$. Solving the optimisation problem in Equation \ref{eqn:opt-prob} with $\Ff$ will be analytically simpler than it would be using $\Ff_0$, and this will provide a tighter surrogate optimisation problem (since $\Ff_0\subset\Ff$, it corresponds to optimising a better bound on the spectral gap), and the solution will be covariant to linear transformations of $\vX(t)$.

\begin{theorem}[$\Lambda = \VVar{X\sim\pi}(X)$ is optimal in terms of lag-$0$ rate of decay of auto\-correlations of linear functions of $\vX$]\label{thm-opt-shape-3}
Suppose that $\pi$ admits second moments. Let $\Gamma = \VVar{X\sim\pi}(X)$. If $\vX$ has generator $G^\Lambda$, then $\Lambda = \Gamma$ maximises the worst case (over $f\in \Ff = \set{x\to v'x \ :\ v\in\RR^k\setminus\set{0} }$) rate at which the auto\-correlation of $f(\vX)$ decays in a neighbourhood of lag-$0$. Thus, in terms of short-run autocorrelations of linear functions of $\vX$, the optimal shaping matrix for RWM proposals is the covariance matrix of the target distribution.
\end{theorem}
\begin{proof}
From Lemma \ref{lem:corr-gen}, for $v\in\RR\setminus\set{0}$ and for $t>0$:
\[
\frac{d}{dt} \Corr(v'\vX(0),v'\vX(0)) \big\given_{t=0^+} 
        & = -\frac{v ' \frac{k\Lambda}{\Lambda:\Sigma} v}{v'\Gamma v}
\]
Hence, we need to solve:
\[
	\max_{\Lambda \succ 0}\min_{v\in\RR^k\setminus\set{0}} \frac{v ' \frac{k\Lambda}{\Lambda:\Sigma} v}{v'\Gamma v}
\]
Substituting $w = \Gamma^{1/2} v$, we can instead solve:
\[
	\max_{\Lambda \succ 0}\min_{w\in\RR^k\setminus\set{0}} \frac{w ' \Gamma^{-1/2} k \Lambda \Gamma^{-1/2} w}{(\Lambda:\Sigma)(w'w)}
		& \equiv \max_{\Lambda \succ 0} \frac{\lambda_1\lcr({\Gamma^{-1/2} k \Lambda \Gamma^{-1/2}})}{\Lambda:\Sigma}
\]
where $\lambda_{i}(A)$ returns the $i$th smallest eigenvalue of $A$. Equivalently, we can solve:
\[
	\min_{\Lambda} \lcr({\lambda_{k}(\Gamma^{1/2}\Lambda^{-1} \Gamma^{1/2})(\Lambda:\Sigma)})
\]
Substituting $\Theta = \Gamma^{1/2}\Lambda^{-1} \Gamma^{1/2}$, we can solve instead:
\[
	\min_{\Theta} \lcr({\lambda_{k}(\Theta)(\Theta^{-1} : (\Gamma^{1/2} \Sigma \Gamma^{1/2}))})
\]

We will solve this optimisation problem by lower bounding the objective function. It will be obvious that $\Theta= I$ achieves the lower bound, and so we will have $\Lambda = \Gamma$ is optimal. The lower bound is given by:
\[
\lcr({\lambda_{k}(\Theta)(\Theta^{-1} : (\Gamma^{1/2} \Sigma \Gamma^{1/2}))})
	& \geq \lambda_{k}(\Theta) \lambda_{1}(\Theta^{-1}) \tr(\Gamma^{1/2} \Sigma \Gamma^{1/2}) \\
	& = \tr(\Gamma^{1/2} \Sigma \Gamma^{1/2})
\]
\end{proof}

This result shows that the rate of convergence of RWM is fundamentally limited by $\frac{k}{\Gamma:\Sigma}$ in the sense that, no matter the choice of proposal shaping matrix, the spectral gap of the generator will be bounded by $\frac{k}{\Gamma:\Sigma}$. When $\Gamma^{-1} = \Sigma$ this leads to the same speed limit as is witnessed by the $\pi\equiv \Nn$ case. When $\Gamma^{-1}$ and $\Sigma$ are very different, this would demonstrate that RWM will be very inefficient no matter how it is tuned.

This can also be used to say, for example, that the rate of convergence of the limiting diffusion, when the proposals are spherical ($\Lambda$= I), cannot be faster that $\frac{k\lambda_1(\Gamma^{-1})}{\tr(\Sigma)} $. On the other hand, using proposals of $\Lambda = \Gamma$, the convergence rate could plausibly be as fast as $\frac{k}{\Sigma:\Gamma}$. Thus, we are somewhat justified to be more optimistic regarding the performance of the shaping $\Lambda = \Gamma$ than in that of the shaping $\Lambda = I$, but we cannot not provided any formal guarantee that the worst case rate of convergence for an arbitrary expected value is actually faster. The very short run performance, however has actually been optimised (by construction), justifying the intuition that optimising the lag-$1$ auto\-correlation is a ``greedy'', suboptimal (but possibly reasonable) approximate solution to the optimal shaping problem.

One may also attempt to address the auto\-correlations for non-linear func\-tions. A particular function of interest is the log-density, $\log\pi(\vX)$, which is the only example we will consider here. Corollary \ref{cor:autocor-logpi}, which follows directly from Lemma \ref{lem:corr-gen}, shows us that for all choices of $\Lambda$ (when the optimal scaling is used) $\log\pi$ has the same rate of decay of the auto\-correlation at lag-$0$. This may be interpreted as saying that the speed of a RWM algorithm is fundamentally limited by the variance of the log-density, uniformly over all possible proposal shaping matrices. That is to say, when $k^{-1}\VVar{X\sim\pi}[\log\pi(X)]$ is very large, then RWM will be inefficient no matter how it is tuned.

\begin{corollary}\label{cor:autocor-logpi} If $\vX(t)$ has generator $G^\Lambda$ then the rate of change of the auto\-correlation of $\log\pi(\vX(t))$ at lag-$0$ is $ \frac{k}{\VVar{X\sim\pi} [\log\pi(X)]}$ uniformly in $\Lambda$.
\end{corollary}
\begin{proof}
This follows directly from Lemma $\ref{lem:corr-gen}$ applied to $\log\pi(X)$.
\end{proof}

\subsection{High Dimensional Dependence Asymptotics}
\label{subsec:hd-dependence}
We now consider the im\-plications of the speed limits derived above on the performance decay for targets with high-dimensional dependence.  This section is intentionally less mathematically rigorous than the rest of the work, with the intention of motivating future research in the area of optimal scaling for high-dimensional targets with non-trivial dependence structures. 

We attempt to use the two ``speed limits'' derived in the previous section to characterise some regimes in which RWM will perform poorly. Consider a sequence of densities of varying dimension $\set{\pi_k}_{k\in\NN}$ with $\pi_k$ a density on $\RR^k$. We consider $\Pi_{k,d} = \pi_k^{\otimes d}$. 

Having chosen to scale time by $dk$ rather than by $d$ the acceleration required to get the weak limit in Theorem \ref{thm:weak-conv} is comparable across targets with equal total dimension, and hence the limiting diffusions should be comparable as well (at least in terms of their spectral gaps and rates of convergence).

If $\VVar{X\sim\pi_k} \log\pi_k (X) \not\in O(k)$ then RWM performance drops off as $k\to\infty$, so dependence structures for which lead to this behaviour are expected to work poorly using RWM (much worse than for an IID target, with a spectral gap tending to $0$), no matter how they are tuned. In the case of IID targets (where $\pi_k =\pi_1 ^{\otimes k}$) we have $\VVar{X\sim\pi_k} \log\pi_k (X) = k \VVar{X\sim\pi_1} \log\pi_1 (X) \in \Theta(k)$. A similar result will hold for rotations and scalings of IID targets), showing that properly shaped proposals will yield commensurate performance for such targets.

The same story holds true if $\Gamma:\Sigma = \VVar{X\sim\pi}[X] : \VVar{X\sim\pi}[\grad\log\pi]\not\in O(k)$. Again, in the case of IID targets (where $\pi_k =\pi_1 ^{\otimes k}$, $\Gamma = \diag(\gamma)$, $\Sigma = \diag(\sigma)$) we have $\Gamma : \Sigma = k\gamma\sigma \in \Theta(k)$. A similar result will hold for rotations and scalings of IID targets), showing that properly shaped proposals will yield commensurate performance for such targets.

While these criteria may not be the sharpest possible, the recipe of (i) deriving a diffusion limit, (ii) deriving upper bounds on the spectral gap using formulas such as Equation \ref{eqn:opt-prob}, and (iii) considering the asymptotics of the upper bound on the spectral gap as the dependence structure tends to infinity can be useful in developing our understanding of the behaviour of MCMC methods for dependent targets (which is an under explored topic in the literature).

\newpage

\section{Proof of Theorem \ref{thm:weak-conv}} \label{sec:pf-weak-conv}
%!TEX root = ../blockiid_main.tex
\setlength{\parindent}{0pt}
\setlength{\parskip}{6pt}
\subsection{Definitions}
We will make consistent use of the results listed in Prop\-osition \ref{prop:summary-of-gll} which hold under our assumptions.

Let $\hat G_d^{l,\Lambda}$ be the generator for a pure jump process with homogeneous jump intensity $kd$, and jump distribution given by the Random Walk Metropolis transition kernel with normal increments of mean $0$ and variance $l^2 I_d\otimes \Lambda / (d-1)$, where $\Lambda$ is symmetric and strictly positive definite. The generator is explicitly given by
\[
  \hat G^{l,\Lambda}_d f (x) 
    &= kd \ \EE_{Z\sim\Nn_d\lcr({l^2\Lambda})} \lcr[{(f(x+Z) - f(x))\lcr({1\wedge\frac{\Pi_d(x+Z)}{\Pi_d(x)}})}] 
\]
Then $\hat G^{l,\Lambda}_d$ is a bounded linear operator on $\hat C(\RR^{kd})$, so we can take its domain to be the full Banach space $\hat C(\RR^{kd})$. 

Let $G^{l,\Lambda}_d$ be the restriction of $\hat G^{l,\Lambda}_d$  to functions of the form $f(x_{1:kd}) = f_1 (x_{1:rk})$ which act only on the first $rk$ components.

Let $G^{l,\Lambda}$ be the generator of an anisotropic Langevin Diffusion,
\[
  G^{l,\Lambda} f
    & = k l^2 a(l) \frac{1}{2}\lcr({[I_r\otimes \Lambda] : \grad^2 f + (\grad\log\pi^{\otimes r})' [I_r\otimes \Lambda] (\grad f)}) 
\]
where $A:B  = \tr(A'B)$, and where $a(l) = 2\Phi\lcr({\frac{-l \sqrt J}{2}})$, and 
\[
	\Sigma 
		& = \Var_{X\sim\pi} [\grad\log\pi(X)] = \EEE{X\sim\pi} \grad\log\pi(X) \grad\log\pi(X)'\\	
	J
		& = \EEE{X\sim\pi}\lcr[{[\grad\log\pi(X)]' \Lambda [\grad\log\pi(X)]}] = \Lambda : \Sigma		\\
\]
We take the domain of $G^{l,\Lambda}$ to be $\Dd(G^{l,\Lambda}) = C^\infty_c(\RR^{rk}) \cap L_2(\RR^{rk},\pi)$.

\subsection{A General Convergence Theorem}
Our goal is to show that if $\vX_d$ has generator $\hat G^{l,\Lambda}_d$ and $\vX^r$   has generator $G^{l,\Lambda}$ then the stochastic process of the first $rk$ components of $\vX_d$, $\vX^{(1):(r)}_d$, converges weakly to $\vX^r$ in the Skorohod topology: $\vX_d^{(1):(r)}\disttosk \vX^r$.

The following result, paraphrased and specialised from \citep{ethier2009markov}, establishes sufficient conditions for this convergence to hold.

\begin{proposition}[Convergence Theorem from \citep{ethier2009markov}]
\label{prop:ek-conv}
Suppose that: 
\begin{itemize}
\item[(i)] $\vX_d$ is a Markov process in $E_d$ with cadlag sample paths and with single-valued full generator $\hat G_d$, and $\vX_d^{(1):(r)} = \rho_d(\vX_d)$ where $\rho_d:E_d\to E$ is measurable. 
\item[(ii)] $G$ is single-valued and its closure generates a Feller semigroup on $E$ corresponding to the Markov process $\vX^r$.
\item[(iii)] The initial distribution of $\vX_d^{(1):(r)}$ converges weakly to the initial dist\-ribution of $\vX^r$; 
\[\vX_d^{(1):(r)}(0) \distto \vX^r(0)\ . \]
\item[(iv)] $\ol{ \Dd(G)}$ contains an algebra which strongly separates points,
\item[(v)] For each $f\in \Dd(G)$, and each $T>0$, is a sequence of functions $f_d\in \Dd(\hat G_d)$, and a sequence of sets $F_d\subset E_d$ such that $	\sup_{d} \norm{f_d} <\infty$, and:
\[	
	\label{eqn:ek-high-prob}
	\lim_{d\to\infty} \PP(\vX_d \in F_d \quad\forall 0\leq t\leq T) & = 1 \\
\]
\[
	\label{eqn:ek-f-conv}
	\lim_{d\to\infty}\sup_{\vx_d\in F_d} \abs{f (\rho_d(\vx_d)) - f_d(\vx_d)} & = 0 \\
\]
\[
	\label{eqn:ek-Gf-conv}
	\lim_{d\to\infty}\sup_{\vx_d\in F_d} \abs{ [Gf](\rho_d(\vx_d)) - [\hat G_d f_d] (\vx_d)} & = 0	
\]
\end{itemize}
Then $\vX_d^{(1):(r)}$ converges weakly in the Skorohod topology to $\vX^r$: $\vX_d^{(1):(r)}\disttosk \vX^r$..

\end{proposition}
\begin{proof}
This is a restatement of \citep{ethier2009markov} Chapter 4, cor\-ollary 8.7. where we have simplified and specialised some of the stated assumptions. In particular, we use that (cadlag $\implies$ progressive), and we assume that all generators involved are single valued, and that $\vX^r$ is Feller which implies its generator must generate a strongly continuous contraction semigroup.
\end{proof}

\begin{remark}
Because $\vX^r$ is assumed to be a Feller process, it has a cadlag modification. Thus without loss of generality, $\vX^r$ may be assumed to be cadlag.
\end{remark}

Thus, taking $E = \RR^{rk}$, and $E_d = \RR^{kd}$, and $\rho_d(x_{1:kd}) = x_{1:rk}$, and $\hat G_d = \hat G^{l,\Lambda}_d$ and $G = G^{l,\Lambda}$ as defined above, we need only verify the five premises of Proposition \ref{prop:ek-conv} in order to establish Theorem \ref{thm:weak-conv}.

\begin{lemma}[Verifying Premises (i)-(iv) of Proposition \ref{prop:ek-conv}]
\label{lem:verify-ek-1-4}
Under the definitions above, and the assumption that $\grad\log\pi$ is $L$-Lipschitz we have that premises (i)-(iv) of Proposition \ref{prop:ek-conv} hold.
\end{lemma}
\begin{proof}
\begin{itemize}
\item[(i)] Since $\hat G_d^{l,\Lambda}$ is a bounded linear operator it must be single valued, and since the domain is the full Banach space $\hat C(\RR^{rkd})$ it must be its own closure. Since it generates a pure jump Markov process (with homogeneous intensity and the RWM transition function), the sample paths of $\vX_d$ must be cadlag.
\item[(ii)] Since $\grad\log\pi$ is assumed to be Lischitz, then by \citep{ethier2009markov} [chapter 8, theorem 2.5], the closure of 
\[
	\set{(f,G^{l,\Lambda} f): f\in C_c^\infty(\RR^{rk})}
\] 
is single valued and generates a Feller semigroup on $\hat C(\RR^{rk})$.
\item[(iii)] This is trivially satisfied because of the assumption that $\vX(0)\sim\pi$ and $\vX_d(0)\sim \Pi_d = \pi^{\otimes d}$
\item[(iv)] In our case, $\hat C(\RR^{rk})\supseteq \ol{ \Dd(G^{l,\Lambda})} \supseteq C_c^\infty(\RR^{rk})$. We prove that the algebra $C_c^\infty(\RR^{rk})$ strongly separates points. Fix $x\in \RR^k$ and $\delta >0$. Consider the location-scale bump function:
\[f_{x,\delta}(y) = \exp\lcr({{-\frac{1}{1-\frac{\norm{y-x}^2}{\delta^2}}}}) \One_{\norm{y-x}<\delta} \ .\]
This function is in $C_c^\infty(\RR^{rk})$ and for $\norm{x-y}>\delta$ we have \[\abs{f(y) -f(x)} \geq 1/e \ .\]
\end{itemize}
\end{proof}

Therefore, to prove Theorem \ref{thm:weak-conv} we need only verify premise (v) of Prop\-osition \ref{prop:ek-conv}. This is done in the next subsection.

\subsection{Verifying Premise (v) of Proposition \ref{prop:ek-conv}}
This premise is more comp\-licated to verify. We first construct the sequence of ``large sets'', $\set{F_d}_{d\in\NN}$ and verify Equation \ref{eqn:ek-high-prob}.  Then, we verify ``uniform convergence of generator evaluations on large sets'', Equation  \ref{eqn:ek-Gf-conv} for $f\in C_c^\infty$ which we have taken to be the domain of $G^{l,\Lambda}$, using $f_d = f\circ \rho_d$. The structure of this section closely follows that of the weak convergence proof in \citep{neal2006optimal}. 

Other than completing the proof in a multivariate setting, we make two notable changes to the structure of the proof relative to \citep{neal2006optimal}. First, we control the size of $\norm{\grad\log\pi^{\otimes r}(x^{(1):(r)})}$ on our ``large set'' by including $F_{d,4}$, which is needed to ensure: 
\[
& \lim_{d\to\infty} \sup_{x\in F_d} \lcr\vert{d \grad\log\pi^{\otimes r}(x^{(1):(r)})'}.\\
&\qquad\qquad \cdot \lcr.{\lcr({\EEE{Z^{(1):(r)}} \lcr[{Z^{(1):(r)}\lcr({h(x^{(1):(r)}+Z^{(1):(r)})- h(x^{(1):(r)})})}] }.}.\\
&\qquad\qquad\qquad\qquad \lcr.{\lcr.{- [I_r\otimes \Lambda] \grad h(x^{(1):(r)})})}\vert \\
& = 0
\]
while \citep{neal2006optimal} appears to only show the equivalent of
\[
& \lim_{d\to\infty} \sup_{x\in F_d} \lcr\vert{kd \EEE{Z^{(1):(r)}} \lcr[{{Z^{(1):(r)}}\lcr({h(x^{(1):(r)}+Z^{(1):(r)}) - h(x^{(1):(r)})})}]}.\\
	&\qquad\qquad \qquad\qquad \lcr.{ - kl^2 \Lambda \grad h(x^{(1):(r)})}\vert\\
	& =  0
\]
which is not sufficient for the final result, since $\grad\log\pi$ may be unbounded and $\EE_{Z^{(1):(r)}} h(x^{(1):(r)} +Z^{(1):(r)})$ is not compactly supported even if $h$ is. 

Secondly, \citep{neal2006optimal} implicitly assumes that the 3rd order partial derivatives of $\pi$ exist and are uniformly bounded. This is needed in their proof to control the 3rd order remainder of a 2nd order Taylor expansion. We circumvent this by including $F_{d,3}$ below in our ``large set'', allowing us to control the relevant error using the convergence of an integrated finite difference to the corresponding derivative. The use of dominated convergence to control the approximation error on this set was inspired by \citep{lalancette2017convergence}, though we have modified the technique to also not require continuous second derivatives.

\begin{remark}
Taking $f_d = f\circ \rho_d$, we have that Equation (\ref{eqn:ek-f-conv}) is trivially satisfied. 
\end{remark}
\subsubsection{Large Sets}\label{subsec:large-sets}
Suppose that $d>r$. The behaviour on the initial segment is irrelevant for the limit.

Let 
\[
  F_d 
     & = F_0^d \cap F_{1,d}\cap F_{2,d} \cap F_{3,d} \cap F_{4,d} \\   
     & = \set{x\in \RR^{kd}: \abs{R_d (x^{(r+1):(d)}) - J} < d^{-1/8} +\frac{J(r-1)}{(d-1)}} \\
     &\qquad\qquad \cap \set {x\in \RR^{kd}: \abs{S_d (x^{(r+1):(d)}) - J} < d^{-1/8} +\frac{J(r-1)}{(d-1)}} \\
     & \qquad\qquad \cap\set{x\in \RR^{kd}: U_d(x^{(r+1):(d)}) \leq \theta(d)+ \sqrt{2Ll^2 K_\Lambda \frac{\log d}{d}} }\\
     &\qquad\qquad \cap \set{x\in \RR^{kd}: \norm{\grad\log\pi^{\otimes r}(x^{(1):(r)})} \leq 2kr\sqrt{L\log d}}\ ,
\] 
where: 
\[
	& F_0 
	 = \set{x\in\RR^k: \grad\log\pi\text{ is differentiable at } x}\\
  	& R_d (x^{(r+1):(d)})
    	  = \frac{1}{d-1} \sum_{i=r+1}^d (\grad\log\pi(x^{(i)})'\Lambda (\grad\log\pi(x^{(i)}))\\
  	& S_d (x^{(r+1):(d)})
    	 = \frac{-1}{d-1} \sum_{i=r+1}^d \Lambda: (\grad^2\log\pi(x^{(i)}))   \\
	& U_d(x^{(r+1):(d)}) \\
    	& \quad = \EEE{Z\sim\Nn_d\lcr({l^2\Lambda})}\abs{\sum_{i=r+1}^d \scriptstyle{\lcr({\log\frac{\pi(x^{(i)}+Z^{(i)})}{\pi(x^{(i)})} - \grad\log\pi(x^{(i)})Z^{(i)} -{Z^{(i)}}' \frac{\grad^2\log\pi(x^{(i)})}{2}Z^{(i)} })}  }\\    	
    &K_\Lambda 
    	= \sqrt{2\norm{\Lambda}_F^2+\tr(\Lambda)^2}\\  
	& \theta(d) = l^2 (d-1) K_\Lambda \EEE{X\sim\pi} \sqrt{\EEE{Z\sim\Nn\lcr({0,\frac{l^2\Lambda}{d-1}})}\scriptstyle{\abs{\log\frac{\pi(X+Z)}{\pi(X)} - \grad\log\pi(X)Z -Z' \frac{\grad^2\log\pi(X)}{2}Z}}^2} \ .
\]

\begin{lemma}[$\lim_{d\to\infty}\theta(d) =  0$]\label{lem:theta-to-0}
Under assumption (A1) $\lim_{d\to\infty}\theta(d) =  0$
\end{lemma}
\begin{proof}
For $W\mysim{iid}\Nn(0,l^2\Lambda)$ and $Z = \frac{W}{\sqrt{d-1}}$, for all $x\in F_0$, using the fundamental theorem of calculus,
\[
	& (d-1)\lcr({\log\frac{\pi(x+Z)}{\pi(x)} - \grad\log\pi(x)Z -{Z}' \frac{\grad^2\log\pi(x)}{2}}) \\ 
		& = \sqrt{d-1} \int_0^1 \lcr({\grad\log\pi\lcr({x+\frac{hW}{\sqrt{d-1}}})- \grad\log\pi(x) })W dh \\
		&\qquad\qquad-  {W}'\frac{\grad^2\log\pi(x)}{2} W\\
		& = \int_0^1 \frac{\grad\log\pi\lcr({x+\frac{hW}{\sqrt{d-1}}})- \grad\log\pi(x) }{1 / \sqrt{d-1}} W dh -  {W}'\frac{\grad^2\log\pi(x)}{2} W . 
\]
For $x\in F_0$, as $d\to\infty$ the integrand converges pointwise to \[W'\grad^2\log\pi(x) W h \] from the differentiability of $\grad\log\pi$ at $x$.
Also, from the Lipschitz property of $\grad\log\pi$, the integrand is bounded by $L h \norm{W}^2$ for all $d\geq 2$ and all $h\in[0,1]$. In fact, we have the following bound which will also be useful later:
\[\label{eqn:theta-to-0-bound}
& \hspace{-2em}\abs{\int_0^1 \frac{\grad\log\pi\lcr({x+\frac{hW}{\sqrt{d-1}}})- \grad\log\pi(x) }{1 / \sqrt{d-1}} W dh -  {W}'\frac{\grad^2\log\pi(x)}{2} W}\\
	& \leq \abs{\int_0^1 \frac{\grad\log\pi\lcr({x+\frac{hW}{\sqrt{d-1}}})- \grad\log\pi(x) }{1 / \sqrt{d-1}} W dh} +\abs{{W}'\frac{\grad^2\log\pi(x)}{2} W}\\
	& \leq \int_0^1 L \norm{W}^2 h\ dh + \frac{L}{2} \norm{W}^2\\
	& = L \norm{W}^2
\]
Therefore, by the bounded convergence theorem,
\[
	&\hspace{-2em} \lim_{d\to\infty}\int_0^1 \frac{\grad\log\pi\lcr({x+\frac{hW}{\sqrt{d-1}}})- \grad\log\pi(x) }{1 / \sqrt{d-1}} W dh\\
		& = \int_0^1 {W}'\grad^2\log\pi(x) W h\ dh \\
		& = {W}'\frac{\grad^2\log\pi(x)}{2} W
\] 
Therefore, for $x\in F_0$, for $W\sim\Nn(0,l^2\Lambda)$,
\[
& \lim_{d\to\infty}(d-1)\lcr({\log\frac{\pi(x+\frac{W}{\sqrt{d-1}})}{\pi(x)} - \grad\log\pi(x)\frac{W}{\sqrt{d-1}}}.\\
& \qquad\qquad\qquad\qquad\qquad \lcr.{-{\frac{W}{\sqrt{d-1}}}' \frac{\grad^2\log\pi(x)}{2} \frac{W}{\sqrt{d-1}}})\\
& =0
\]
Now, using the bound in Equation \ref{eqn:theta-to-0-bound} again to upper bound the integrand, by the dominated convergence theorem, we have (with $W\sim\Nn(0,l^2\Lambda)$):
\[
	& \lim_{d\to\infty}(d-1)^2\EEE{W} \lcr\vert{\lcr({\log\frac{\pi(x+\frac{W}{\sqrt{d-1}})}{\pi(x)} - \grad\log\pi(x)\frac{W}{\sqrt{d-1}} }.}.\\
	&\qquad\qquad\qquad\qquad\lcr.{\lcr.{-{\frac{W}{\sqrt{d-1}}}' \frac{\grad^2\log\pi(x)}{2} \frac{W}{\sqrt{d-1}}})}\vert^2\\
	& =0 \ .
\]
As a function of $x$, this is uniformly bounded by $L^2 \EEE{W}  \norm{W}^4 <\infty$ on $F_0$, so applying the dominated convergence theorem again, since $F_0$ has measure $1$ under $\pi$ (since $\pi$ is absolutely continuous with respect to the Lebesgue measure) we have:
\[
\theta(d) 
    & = l^2 (d-1) K_\Lambda \EEE{X\sim\pi} \sqrt{\EEE{Z\sim\Nn\lcr({0,\frac{l^2\Lambda}{d-1}})}\scriptstyle{\abs{\log\frac{\pi(X+Z)}{\pi(X)} - \grad\log\pi(X)Z -Z' \frac{\grad^2\log\pi(X)}{2}Z}^2}} \\
    & \to 0
\]
\end{proof}

\begin{lemma} If $\vX_d$ is the pure jump process with generator $G^{l,\Lambda}_d$,  and if $\vX_d(0)\sim \Pi_d$, then for any $T>0$
  \[\PP(\vX_d(t) \in F_d\quad  \forall 0\leq t\leq T) \to 1\]
\end{lemma}
\begin{proof}
Since the number of possible jumps times of the process $\vX $ in the interval $[0,T]$, $N_T$, is distributed as $N_T\sim \poisson(Tkd)$, we have
\[
  & \PP(\text{not }\lcr({\vX_d(t) \in F_d\quad  \forall 0\leq t\leq T}))\\
    & = \EE\EE\lcr[{\one_{\exists t\in[0,T]:\ \vX_d(t) \not\in F_d} \given N_T}]\\
    & \leq \EE\EE\lcr[{\one_{X^{(1):(d)}_0\not\in F_d} +  \sum_{\substack{t>0\\ N_t\neq N_{t^-}}} \one_{\vX_d(t) \not\in F_d\quad} \bigg\given N_T}]            \\
    & = \EE\lcr[{ (N_T+1) \Prr{X_d\sim \Pi_d} (X_d\not\in F_d)}] \\ 
    & = (Tkd+1) \Prr{X_d\sim \Pi_d} (X_d\not\in F_d)
\]
Thus it is sufficient to show that $\Prr{X_d\sim \Pi_d} (X_d\not\in F_d) = o(1/d)$. Applying subadditivity:
\[
  & \Prr{X_d\sim \Pi_d} (X_d\not\in F_d) \\
    & \leq \Prr{X_d\sim \Pi_d}\lcr({\abs{R_d (X_d^{(r+1):(d)}) -J} > d^{-1/8} +\frac{J(r-1)}{(d-1)}}) \\
    & \qquad\qquad + \Prr{X_d\sim \Pi_d}\lcr({\abs{S_d (X^{(r+1):(d)}) -J} > d^{-1/8} +\frac{J(r-1)}{(d-1)}}) \\ 
    & \qquad\qquad +  \Prr{X_d\sim \Pi_d}\lcr({U_d(X^{(r+1):(d)}) > \theta(d)+ \sqrt{2Ll^2 K_\Lambda \frac{\log d}{d}}}) \\
    & \qquad\qquad + \Prr{X_d\sim \Pi_d} \lcr({\norm{\grad\log\pi^{\otimes r}(X^{(1):(r)})} > 2Lkr\sqrt{\log d}})\ ,
\]
so it is sufficient to show that each of these four terms is individually $o(1/d)$.

It is obvious that $\EE_{X_d\sim \Pi_d} R_d (X_d^{(r+1):(d)})  = J \frac{d-r}{d-1}$. From distributional integ\-ration by parts (Theorem \ref{thm:ibps}) we also have $\EE_{X\sim \Pi_d} S_d (X^{(r+1):(d)})  = J\frac{d-r}{d-1}$.

For the first term, since $\grad\log\pi(Y)$ is subgaussian for $Y\sim\pi$ (see Lemma \ref{lem:gll-subg}), and hence has all of its polynomial moments:
\[ \scriptstyle
& \hspace{-2em}\Prr{X\sim \Pi_d}\lcr({\abs{R_d (X^{(r+1):(d)}) -J} > d^{-1/8} +\frac{J(r-1)}{(d-1)}})\\
    &\scriptstyle\leq \EEE{X\sim \Pi_d}\lcr({\abs{R_d (X^{(r+1):(d)}) -\EE_{X_d\sim \Pi_d} R_d (X_d^{(r+1):(d)}) }^4}) d^{1/2} \\ 
    &\scriptstyle = d^{1/2} \frac{\EEE{Y\sim \pi}\lcr({A(Y)^4}) + 3(d-2) \EEE{Y\sim \pi}\lcr({A(Y)^2})^2}{(d-1)^3} \\
    &\scriptstyle\leq \frac{6 }{(d-1)^{3/2}} \norm{\Lambda}^4 M_1
\]
for $M_1<\infty $ sufficiently large, where $A(Y) = (\grad\log\pi(Y)'\Lambda (\grad\log\pi(Y)) -J$.

For the second term, again since $\grad\log\pi(Y)$ has all of its polynomial mo\-ments for $Y\sim\pi$:
\[
&\scriptstyle\hspace{-2em} \Prr{X\sim \Pi_d}\lcr({\abs{S_d (X^{(r+1):(d)}) -J} > d^{-1/8} +\frac{J(r-1)}{(d-1)}})\\
    &\scriptstyle\leq \EEE{X\sim \Pi_d}\lcr({\abs{S_d (X^{(r+1):(d)}) -\EE_{X\sim \Pi_d} S_d (X^{(r+1):(d)})}^4}) d^{1/2} \\ 
    &\scriptstyle = d^{1/2} \frac{\EEE{Y\sim \pi}\lcr({B(Y)^4}) + 3(d-2) \EEE{Y\sim \pi}\lcr({B(Y)^2})^2}{(d-1)^3} \\
    &\scriptstyle \leq \frac{6 }{(d-1)^{3/2}} \norm{\Lambda}_F ^4 M_2
\]
for $M_2<\infty $ sufficiently large, where $B(Y) =\Lambda : (\grad^2\log\pi(Y)) - J$.

For the third term (letting $W\sim\Nn\lcr({0, l^2\Lambda})$ ):
\[
  & U_d(x^{(r+1):(d)}) \\
    & = \EEE{Z\sim\Nn_d\lcr({l^2\Lambda})}\abs{\sum_{i=r+1}^d \scriptstyle{\lcr({\log\frac{\pi(x^{(i)}+Z^{(i)})}{\pi(x^{(i)})} - \grad\log\pi(x^{(i)})Z^{(i)} -{Z^{(i)}}' \frac{\grad^2\log\pi(x^{(i)})}{2}Z^{(i)} })}  }\\
    &\leq \EEE{Z\sim\Nn_d\lcr({l^2\Lambda})}\sum_{i=r+1}^d \scriptstyle{\abs{\lcr({\log\frac{\pi(x^{(i)}+Z^{(i)})}{\pi(x^{(i)})} - \grad\log\pi(x^{(i)})Z^{(i)} -{Z^{(i)}}' \frac{\grad^2\log\pi(x^{(i)})}{2}Z^{(i)} })}  }\\
    & = \frac{1}{d-1}\sum_{i=r+1}^d \EE_W  \abs{\int_0^1 \scriptstyle{\frac{\grad\log\pi\lcr({x^{(i)}+\frac{hW}{\sqrt{d-1}}})- \grad\log\pi(x^{(i)}) }{1 / \sqrt{d-1}} W dh -  {W}'\frac{\grad^2\log\pi(x^{(i)})}{2} W}}\\
    &\leq \frac{1}{d-1}\sum_{i=r+1}^d \lcr({{\scriptstyle \EE_W  \norm{W}^4}})^{\frac{1}{2}} \\
    &\qquad \times\lcr({{\scriptstyle\EE_W} \abs{\frac{1}{{\scriptstyle\norm{W}^2}}{\int_0^1 {\scriptstyle \frac{\grad\log\pi\lcr({x^{(i)}+\frac{hW}{\sqrt{d-1}}})- \grad\log\pi(x^{(i)}) }{1 / \sqrt{d-1}} W dh -  {W}'\frac{\grad^2\log\pi(x^{(i)})}{2} W}}}^2})^{\frac{1}{2}}
\]
Using Isserlis' theorem (\citep{isserlis1916on}, equation (39)* therein)
\[
\EEE{W\sim\Nn\lcr({0, l^2\Lambda})}\norm{W}^4
  & = \EEE{W\sim\Nn\lcr({0, l^2\Lambda})}\lcr({\sum_{i=1}^k W_i^2})^2 
  && = \sum_{i=1}^k\sum_{j=1}^k \EEE{W\sim\Nn\lcr({0, l^2\Lambda})} [W_i^2 W_j^2]\\
  & = \sum_{i=1}^k\sum_{j=1}^k l^4 (\Lambda_{ii}\Lambda_{jj} +2 \Lambda_{ij}^2)
  && = l^4 \lcr({2\norm{\Lambda}_F^2 + \tr(\Lambda)^2})
  \]
Thus:
\[
& U_d(x^{(r+1):(d)}) \\
 	&\leq V_d(x^{(r+1):(d)}) \\
  	&: = \mbox{\normalsize $\frac{l^2K_\Lambda}{(d-1)}$} \sum_{i=r+1}^d \lcr({{\scriptstyle\EE_W} \abs{\frac{\int_0^1 {\scriptstyle \frac{\grad\log\pi\lcr({x^{(i)}+\frac{hW}{\sqrt{d-1}}})- \grad\log\pi(x^{(i)}) }{1 / \sqrt{d-1}} W dh -  {W}'\frac{\grad^2\log\pi(x^{(i)})}{2} W}}{{\scriptstyle\norm{W}^2}}}^2})^{\frac{1}{2}}
\]
and so:
\[
  \EE_{X\sim\Pi_d} V_d(X) = \theta(d) \frac{d-r}{d-1} \leq \theta(d)
\]  
Moreover, from Equation \ref{eqn:theta-to-0-bound}, for $x\in F_0$:
\[
\lcr({{\scriptstyle\EE_W} \abs{\frac{\int_0^1 {\scriptstyle \frac{\grad\log\pi\lcr({x^{(i)}+\frac{hW}{\sqrt{d-1}}})- \grad\log\pi(x^{(i)}) }{1 / \sqrt{d-1}} W dh -  {W}'\frac{\grad^2\log\pi(x^{(i)})}{2} W}}{{\scriptstyle\norm{W}^2}}}^2})^{\frac{1}{2}} \leq L
\]
Thus, by Hoeffding's Inequality (\citep{boucheron2013concentration}, theorem 2.8 therein):
\[
&\hspace{-2em}\Prr{X_d\sim \Pi_d}\lcr({U_d(X_d^{(r+1):(d)}) > \theta(d) + \sqrt{2l^2 L K_\Lambda \frac{\log d}{d}}}) \\
  & \leq \Prr{X_d\sim \Pi_d}\lcr({V_d(X_d^{(r+1):(d)}) > \theta(d) + \sqrt{2 l^2 L K_\Lambda \frac{\log d}{d}}}) \\
  & \leq \frac{1}{d^2}
\]

For the fourth (and last) term, since, $\grad\log\pi(X)$ is subgaussian with proxy variance $L$ for $X\sim\pi$ (see Lemma \ref{lem:gll-subg}), then $\norm{\grad\log\pi(X)}$ is subgaussian with proxy variance $L k^2$. Thus
\[
\Prr{X_d\sim \Pi_d} \lcr({\norm{\grad\log\pi(X_d^{(1)})} > 2k\sqrt{L\log d}}) \leq \frac{2}{d^2}
\]
Now, for $\lcr({\norm{\grad\log\pi^{\otimes r}(X_d^{(1):(r)})} > 2kr\sqrt{L\log d}})$ to occur, at least one block, indexed by $j\in\set{1,...,r}$, must have $\lcr({\norm{\grad\log\pi(X_d^{(j)})} > 2k\sqrt{L\log d}})$. Thus, 
\[
\Prr{X_d\sim \Pi_d} \lcr({\norm{\grad\log\pi^{\otimes r}(X_d^{(1):(r)})} > 2rk\sqrt{L\log d}}) \leq \frac{2 r}{d^2}
\]

Thus:
\[
  1- \PP(\vX_d(t) \in F_d\quad  \forall 0\leq t\leq T)
    & \leq (Tkd+1) \Prr{X_d\sim \Pi_d} (X_d\not\in F_d) \to 0
\]

\end{proof}
\subsubsection{Uniform Convergence of Generator Evaluations on Large Sets}
For each $h\in C(\RR^k)$ let $h_d = h\circ \rho_d$. For the remainder of this section, $Z\sim \Nn\lcr({0, \frac{l^2}{(d-1)} I_d\otimes\Lambda})$ unless stated otherwise, and $Z^{(1)} \sim \Nn_d\lcr({l^2 \Lambda})$ is the first $k$ component block of $Z$.

We introduce an intermediate object, $\tilde G_d^{l,\Lambda}$ on $\set{h\circ \rho_d : h\in C_c^\infty}$, which res\-embles, but is not, a generator for a diffusion process. Take $\tilde G_d^{l,\Lambda}$ given by:
\[
  \tilde G_d^{l,\Lambda} h_d (x) 
    & = \frac{k l^2}{2} \EE_Z[ 1\wedge e^{B_d(x,Z^{(r+1):(d)})}]\ [I_r\otimes\Lambda] : \grad^2 h(x^{(1)}) \\
    &\qquad\qquad + k l^2 \EE_Z[ 1\wedge e^{B_d(x,Z^{(r+1):(d)})}; B_d(x,Z^{(r+1):(d)})<0]\  \\
    &\qquad\qquad\qquad\times (\grad\log\pi^{\otimes r} (x^{(1):(r)}))' [I_r\otimes\Lambda] (\grad h(x^{(1):(r)}))
\]
where 
\[
	B_d(x,Z^{(r+1):(d)}) = \sum_{i=r+1}^d \epsilon(x^{(i)},Z^{(i)})
\] 
and 
\[
	\epsilon(x^{(i)},Z^{(i)}) = \log\frac{\pi(x^{(i)}+Z^{(i)})}{\pi(x^{(i)})} \ .
\] 

We will show that for any $h\in C_c^\infty(\RR^rk)$ we have both:
\[
  \lim_{d\to\infty}\sup_{x\in F_d} \abs{\hat G_d^{l,\Lambda} h_d(x) - \tilde G_d^{l,\Lambda} h_d (x)} & = 0
\]
which is verified in Lemma \ref{lem-rn06-a3}, and
\[
  \lim_{d\to\infty}\sup_{x\in F_d} \abs{\tilde G_d^{l,\Lambda} h_d(x) - G^{l,\Lambda} h (x)} & = 0 
\]
which is verified through Lemma \ref{lem-rn06-a6}.

Then, since $G^{l,\Lambda} h (x^{(1)}) = [G^{l,\Lambda} h]\circ \rho_d (x)$, we will have verified Equation \ref{eqn:ek-Gf-conv}.

\begin{lemma}[$\tilde G_d^{l,\Lambda}$ is close to $\hat G_d^{l,\Lambda}$]\label{lem-rn06-a3}
\[
  \lim_{d\to\infty}\sup_{x\in F_d} \abs{\hat G_d^{l,\Lambda} h_d(x) - \tilde G_d^{l,\Lambda} h_d (x)} = 0
\]
\end{lemma}
\begin{proof}
Notice that: 
\[
  & \hat G_d^{l,\Lambda} h_d (x)\\
    & = \scriptstyle k d \EEE{Z^{(1):(r)}} \lcr[{\lcr({h(x^{(1):(r)}+Z^{(1):(r)}) - h(x^{(1):(r)})})\EEE{Z^{(r+1):(d)}}\lcr({1\wedge \prod_{i=1}^d \frac{\pi(x^{(i)}+Z^{(i)})}{\pi(x^{(i)})}})}] \ .
\]

Then letting
\[
	\Ee(z^{(1):(r)}, x) = \sum_{j=1}^r\epsilon(x^{(j)},z^{(j)}) \ ,
\]
and
\[
  \gamma(z^{(1):(r)}, x)
    & = \EEE{Z^{(r+1):(d)}}\lcr({1\wedge e^{\Ee(z^{(1):(r)}, x)+B_d(x,Z^{(r+1):(d)})}})\ ,
\]

from the integral form of Taylor's remainder theorem (which is valid as long as the derivative occurring in the integral remainder is defined in a weak sense), we have that:
\[
  &\gamma(z^{(1):(r)}, x) \\
    & =  \EEE{Z^{(r+1):(d)}}\lcr({1\wedge e^{B_d(x,Z^{(r+1):(d)})}})\\
    & \quad + {z^{(1):(r)}} '\grad\log\pi^{\otimes r}(x^{(1):(r)}) {\scriptstyle\EEE{Z^{(r+1):(d)}}\lcr({e^{B_d(x,Z^{(r+1):(d)})}; B_d(x,Z^{(r+1):(d)})<0})}\\
    & \quad + \int_0^1 (1-\eta) \lcr[{{z^{(1):(r)}} '\grad^2\log\pi^{\otimes r}(x^{(1):(r)}+\eta z^{(1):(r)}) z^{(1):(r)} }.\\
    &\qquad\qquad \lcr.{+ ({z^{(1):(r)}} '\grad\log\pi^{\otimes r}(x^{(1):(r)}+\eta z^{(1):(r)}))^2}] \\
    &\qquad\qquad {\scriptstyle \times\EEE{Z^{(r+1):(d)}}\lcr({e^{\Ee(z^{(1):(r)}, x) + B_d(x,Z^{(r+1):(d)})}; \Ee(z^{(1):(r)}, x) +B_d(x,Z^{(r+1):(d)})<0})} \ d\eta
\]
Thus,
\[
  &\hat G_d h_d (x)\\
    & = {\scriptstyle kd \EEE{Z^{(1):(r)}} \lcr[{\lcr({h(x^{(1):(r)}+Z^{(1):(r)}) - h(x^{(1):(r)})})\gamma(z^{(1):(r)}, x)  }]}\\
%%%%%%%%%%%%%%%%%%%%%%%%%%%%%%%%%%%%%%%%%%%%%%%%%%%%%%%%%%%%%%%%%%%%%%%%%%%%%%%%%%%%%%	
	& {\scriptstyle= kd \EEE{Z^{(1):(r)}} \lcr[{\lcr({h(x^{(1):(r)}+Z^{(1):(r)}) - h(x^{(1):(r)})})}] \lcr[{\EEE{Z^{(r+1):(d)}}\lcr({1\wedge e^{B_d(x,Z^{(r+1):(d)})}}) }]}\\
	&\qquad {\scriptstyle+ kd \grad\log\pi^{\otimes r}(x^{(1):(r)})' \EEE{Z^{(1):(r)}} \lcr[{{Z^{(1):(r)}}\lcr({h(x^{(1):(r)}+Z^{(1):(r)}) - h(x^{(1):(r)})})}] }\\
	&\qquad\qquad {\scriptstyle\times \EEE{Z^{(r+1):(d)}}\lcr({e^{B_d(x,Z^{(r+1):(d)})}; B_d(x,Z^{(r+1):(d)})<0})}\\	
    & \qquad {\scriptstyle+ kd \int_0^1 (1-\eta) \EEE{Z^{(1):(r)}}\lcr[{\lcr({h(x^{(1):(r)}+Z^{(1):(r)}) - h(x^{(1):(r)})}) }.}\\	
    &\qquad\qquad\qquad {\scriptstyle\times\lcr.{\lcr({{Z^{(1):(r)}} '\grad^2\log\pi^{\otimes r}(x^{(1):(r)}+\eta Z^{(1):(r)}) Z^{(1):(r)}}.}. }\\
    &\qquad\qquad\qquad\qquad {\scriptstyle\lcr.{\lcr.{+ ({Z^{(1):(r)}} '\grad\log\pi^{\otimes r}(x^{(1):(r)}+\eta Z^{(1):(r)}))^2})}] }\\
    &\qquad\qquad {\scriptstyle\times\lcr.{\EEE{Z^{(r+1):(d)}}\lcr({e^{\Ee(z^{(1):(r)}, x) + B_d(x,Z^{(r+1):(d)})}; \Ee(z^{(1):(r)}, x)+B_d(x,Z^{(r+1):(d)})<0})  \ d\eta}]}\\
\]
Thus, since $1\wedge e^{...} \in [0,1]$, using the bound of $\norm{\grad\log\pi^{\otimes r}}$ for $x\in F_{d,4}$:
\[
	&\sup_{x\in F_d}\abs{\hat G_d h_d - \tilde G_d h_d}\\
		& \leq {\scriptstyle \sup_{x\in F_d} \lcr\vert{kd \EE_Z [h(x^{(1):(r)}+Z^{(1):(r)})-h(x^{(1):(r)})] - k \frac{l^2}{2} [I_r\otimes\Lambda]: \grad^2 h(x^{(1):(r)})}\vert} \\
		&\qquad + {\scriptstyle 2kr L^{1/2} \sup_{x\in F_d} \lcr[{\sqrt{\log d} \lcr\Vert{kd\EE_Z [Z^{(1):(r)}(h(x^{(1):(r)}+Z^{(1):(r)})-h(x^{(1):(r)}))]}.}. }\\
		&\qquad\qquad\qquad\qquad\qquad\qquad {\scriptstyle - \lcr.{\lcr.{kl^2 [I_r\otimes\Lambda] (\grad h(x^{(1):(r)}))}\Vert}]}\\
    & \qquad {\scriptstyle + \sup_{x\in F_d}\lcr\vert{kd\int_0^1 (1-\eta)\EEE{Z^{(1):(r)}}\lcr[{\lcr({h(x^{(1):(r)}+Z^{(1):(r)}) - h(x^{(1):(r)})}) }.}.}\\
    &\qquad\qquad\qquad\qquad{\scriptstyle \lcr.{\lcr.{ \times \lcr({{Z^{(1):(r)}} '\grad^2\log\pi^{\otimes r}(x^{(1):(r)}+\eta Z^{(1):(r)}) Z^{(1):(r)} }.}.}.}\\
    &\qquad\qquad\qquad\qquad {\scriptstyle \lcr.{\lcr.{\lcr.{ + ({Z^{(1):(r)}} '\grad\log\pi^{\otimes r}(x^{(1):(r)}+\eta Z^{(1):(r)}))^2})}]\ d\eta }\vert}    
\]

By Lemma \ref{lem-nr06-a2}, the first terms both go to $0$ as $d\to\infty$. Thus we need only handle the last (remainder) term for $x\in F_d$. We have (by Taylor's Theorem):
\[
	&\scriptstyle
		\int_0^1 (1-\eta) \EE_{Z^{(1):(r)}}\lcr\vert{\lcr({h(x^{(1):(r)}+Z^{(1):(r)}) - h(x^{(1):(r)})})}. \\
		&\qquad \qquad \scriptstyle \lcr.{ \times \lcr({{Z^{(1):(r)}} '\grad^2\log\pi^{\otimes r}(x^{(1):(r)}+\eta Z^{(1):(r)}) Z^{(1):(r)}}.}. \\
		&\qquad\qquad\qquad\qquad \scriptstyle \lcr.{\lcr.{+ ({Z^{(1):(r)}} '\grad\log\pi^{\otimes r}(x^{(1):(r)}+\eta Z^{(1):(r)}))^2})}\vert d\eta\\
	& \scriptstyle \leq  \frac{1}{2}\EE_{Z^{(1):(r)}}\lcr({\sup_{y\in\RR^{rk}} \norm{\grad h(y)}})\norm{Z^{(1):(r)}}^3 \\
	& \qquad \qquad \scriptstyle \times\lcr({L + \lcr({\norm{\grad\log\pi^{\otimes r}(x^{(1):(r)})} + \norm{Z^{(1):(r)}} L})^2})\\
	&\scriptstyle \myeqm{x\in F_d} O(d^{-3/2}\log(d)^2)		
\]
Thus the final term in the bound on 
\[
	\sup_{x\in F_d}\abs{\hat G_d h_d - \tilde G_d h_d} 
\]
is $O(d^{-1/2} \log(d)^2)$, so the whole bound goes to $0$.
\end{proof}

\begin{lemma}[$\tilde G_d^{l,\Lambda}$ is close to $G^{l,\Lambda}$]\label{lem-rn06-a6}
\[
\lim_{d\to\infty}\sup_{x\in F_d} \abs{2\Phi(-l\sqrt{I} /2)  - \EEE{Z^{(r+1):(d)}}\lcr[{1\wedge e^{B_d(x,Z^{(r+1):(d)})}}] } = 0
\]
and:
\[
\lim_{d\to\infty}\sup_{x\in F_d} \abs{\Phi(-l\sqrt{I} /2)  - \EEE{Z^{(r+1):(d)}}\lcr[{e^{B_d(x,Z^{(r+1):(d)})}; B_d(x,Z^{(r+1):(d)})<0}] } = 0
\]
and hence:
\[
  \lim_{d\to\infty}\sup_{x\in F_d} \abs{\tilde G_d^{l,\Lambda} h_d(x) - G^{l,\Lambda} h (x^{(1):(r+1)})} & = 0 
\]

\end{lemma}
\begin{proof}
Let 
\[
	& A_d(x,Z^{(r+1):(d)}) \\
	& = \sum_{i=r+1}^d \lcr[{\grad\log\pi(x^{(i)})' Z^{(i)} -\frac{l^2}{2(d-1)} \grad\log\pi(x^{(i)})' \Lambda \grad\log\pi(x^{(i)})}]\]
and let 
\[
	& W_d (x^{(1):(d)}) \\
		& =  \frac{1}{2} \sum_{i = r+1}^d \scriptstyle\lcr[{ {Z^{(i)}} ' [\grad^2 \log \pi(x^{(i)})] Z^{(i)} + \frac{l^2}{(d-1)} (\grad \log\pi(x^{(i)}))' \Lambda (\grad \log\pi(x^{(i)}))}]
\]

Thus, since $y\mapsto 1\wedge e^{y}$ is 1-Lipschitz,
\[
&\hspace{-2em}
\abs{\EEE{Z^{(r+1):(d)}} \lcr[{1\wedge e^{A_d(x,Z^{(r+1):(d)})}}] - \EEE{Z^{(r+1):(d)}}\lcr[{1\wedge e^{B_d(x,Z^{(r+1):(d)})}}] } \\
	& \leq \EE\abs{W_d(x)}+ U_d(x) \
\]
where $U_d$ is defined in the Subsection \ref{subsec:large-sets}. 

Let $\phi_d= \sup_{x\in F_d}\lcr({\EE\abs{W_d(x)}+ \theta(d) + \sqrt{2Ll^2 K_\Lambda\frac{\log d}{d}}})$. By Lemma \ref{lem:theta-to-0} and Lemma \ref{lem-Tr-Err-bd}, $\phi_d\to 0$.

For the second result, let 
\[
	& q_d(x,Z^{(r+1):(d)}) \\
	& \qquad \scriptstyle = \lcr({e^{A_d(x,Z^{(r+1):(d)})}; A_d(x,Z^{(r+1):(d)})<0}) - \lcr({e^{B_d(x,Z^{(r+1):(d)})}; B_d(x,Z^{(r+1):(d)})<0})
\]
and let $\delta_d = \sqrt{\phi_d}$. For simplicity in the rest of the proof, we abbreviate $q_d(x,Z^{(r+1):(d)})$, $A_d(x,Z^{(r+1):(d)})$, $B_d(x,Z^{(r+1):(d)})$ as $q_d, A_d, B_d$ respectively.
\[
	\EE_Z \abs{q_d} \leq \delta_d \PP(\abs{q_d}\leq \delta_d) + \PP(\abs{q_d}> \delta_d)
\]
The first term is $O(\delta_d)$, uniformly in $x$, so its $\sup_{x\in F_d}$ goes to $0$.

The second term can be bounded as:
\[
	\PP_Z(\abs{q_d}> \delta_d)
		& = \PP_Z(\abs{q_d}> \delta_d;\ A_d(x,Z)<0;\ B_d <0)\\
			&\qquad  + \PP_Z(\abs{q_d}> \delta_d;\ A_d \geq 0;\ B_d < 0)\\
			&\qquad  + \PP_Z(\abs{q_d}> \delta_d;\ A_d < 0;\ B_d \geq 0)\\
		& \leq \PP_Z(\abs{A_d-B_d}> \delta_d;\ A_d<0;\ B_d <0)\\
			&\qquad  + \PP_Z(A_d \geq 0;\ B_d < 0)+\PP_Z(A_d < 0;\ B_d \geq 0)\\
		& \leq \PP_Z(\abs{A_d-B_d}> \delta_d;\ A_d<0;\ B_d <0)\\
			&\qquad  + \PP_Z(\abs{A_d-B_d}> \delta_d;\ A_d \geq 0;\ B_d < 0)\\
				&\qquad\qquad+\PP(\abs{A_d-B_d}> \delta_d;\ A_d < 0;\ B_d \geq 0)\\
			&\qquad  + \PP_Z(\abs{A_d-B_d}\leq \delta_d;\ A_d \geq 0;\ B_d < 0)\\
				&\qquad\qquad+\PP_Z(\abs{A_d-B_d}\leq \delta_d;\ A_d < 0;\ B_d \geq 0)\\
		& \leq \PP_Z(\abs{A_d-B_d}> \delta_d)\\
			&\qquad  + \PP_Z(-\delta_d \leq A_d\leq \delta_d)		
\]
By Markov's Inequality, uniformly in $x\in F_d$
\[\PP_Z(\abs{A_d-B_d}> \delta_d) \leq \frac{1}{\delta_d}\phi_d \leq\sqrt{\phi_d}\] 

Next, Since $A_d\sim \Nn(-l^2 R_d/2, l^2 R_d)$, and since $\abs{R_d - J} \leq d^{-1/8} + J\frac{r-1}{d-1}$ on $F_d$, then, since $J> 0$ we have for all $\epsilon>0$, (for $d$ sufficiently large),  
\[
	\sup_{x\in F_d} \PP(-\delta_d < A_d <\delta_d) \leq \epsilon + \PP \lcr[{\Nn(-J/2, J)\in (-\delta_d , \delta_d)}]
\]
Thus, $\lim_{d\to\infty} \sup_{x\in F_d} \PP(-\delta_d < A_d <\delta_d) = 0$.

Now, since $A_d\sim \Nn(-l^2 R_d/2, l^2 R_d)$, by Proposition \ref{prop-censored-LN-mean}:
\[
	\EE[1\wedge e^{A_d}] 
		& = 2\Phi(-l \sqrt{R_d}/2 )
\]
Thus, because $J>0$, and since $y\mapsto \Phi(-l \sqrt{y}/2 )$ is uniformly continuous on $[J-d^{-1/8}- J\frac{r-1}{d-1}, J+d^{-1/8}+ J\frac{r-1}{d-1}]$ for $d$ sufficiently large, we have:
\[	
	& \lim_{d\to\infty}\sup_{x\in F_d} \abs{\EE[1\wedge e^{A_d}] - 2\Phi(-l\sqrt{J} /2) } \\
	& = \lim_{d\to\infty}\sup_{x\in F_d} \abs{ 2\Phi(-l \sqrt{R_d}/2 ) - 2\Phi(-l\sqrt{J} /2) }\\
	& = 0
\]
Analogously, for the truncated expectation:
\[	
	& \lim_{d\to\infty}\sup_{x\in F_d} \abs{\EE[1\wedge e^{A_d}; A_d<0] - \Phi(-l\sqrt{J} /2) } \\
	& = \lim_{d\to\infty}\sup_{x\in F_d} \abs{ \Phi(-l \sqrt{R_d}/2 ) - \Phi(-l\sqrt{J} /2) } \\
	& = 0
\]

Finally, since $h$ has compact support, and since the functions  $\norm{\grad^2 h(x)}_F$ and $\norm{\grad\log\pi^{\otimes r}(x)} \norm{\grad h(x)}$ are both continuous, then they are both uniformly bounded over $x\in\RR^k$ by $M_h<\infty$, say. Hence:
\[
  	& \lim_{d\to\infty}\sup_{x\in F_d} \abs{\tilde G_d^{l,\Lambda} h_d(x) - G^{l,\Lambda} h (x)} \\
  		& \quad \leq \scriptstyle \lim_{d\to\infty}\sup_{x\in F_d} \frac{kl^2}{2} \abs{2\Phi(-l\sqrt{I} /2)  - \EEE{Z^{(r+1):(d)}}\lcr[{1\wedge e^{B_d}}] }  \norm{\Lambda}_F \norm{\grad^2 h(x)}_F \\
  		&\qquad\qquad\qquad\qquad \scriptstyle  + kl^2 \abs{\Phi(-l\sqrt{I} /2)  - \EEE{Z^{(r+1):(d)}}\lcr[{e^{B_d}; B_d<0}] } \norm{\grad\log\pi(x)} \norm\Lambda \norm{\grad h(x)}\\
%%%%%%%%%%%%%%%%%%%%%%%%%%%%%%%%%%%  		
  		&\quad \leq \scriptstyle  M_h \lim_{d\to\infty}\sup_{x\in F_d} \frac{kl^2}{2} \abs{2\Phi(-l\sqrt{I} /2)  - \EEE{Z^{(r+1):(d)}}\lcr[{1\wedge e^{B_d}}] }  \\
  		&\qquad\qquad\qquad\qquad \scriptstyle  + kl^2 \abs{\Phi(-l\sqrt{I} /2)  - \EEE{Z^{(r+1):(d)}}\lcr[{e^{B_d}; B_d<0}] } \\
%%%%%%%%%%%%%%%%%%%%%%%%%%%%%%%%%%%%
		& \quad = 0
\]
\end{proof}

\newpage
\section{Distributional Integration By Parts and Properties of Grad-Log-Lipschitz Probability Densities}\label{sec:ibps}
%!TEX root = ../blockiid_main.tex

\setlength{\parindent}{0pt}
\setlength{\parskip}{6pt}
\subsection{Distributional Integration by Parts}
\label{subsec:ibps}
For this section we adopt the not\-ation of geometric measure theory, where $d \Ll^a x$ denotes integration with respect to the $a$-dimensional Lebesgue measure over the variable $x$, and $d \Hh^b y$ denotes integration with respect to the $b$-dimensional Hausdorff measure over the variable $y$.

\begin{theorem}[Distributional Integration by Parts]\label{thm:ibps}
Suppose that $\pi$ is a continuous, $\Ll^k$-almost-everywhere differentiable  probability density on $\RR^k$ and that $\norm{\grad \pi}$ is $\Hh^{k-1}$-essentially bounded on $\Ll^1$-almost-every level set of $\pi$ (over points where it exists). Then, for any $f: \RR^k\to \RR$ which is locally Lipschitz, with $\grad f(x)$ and $f(x)\grad\log\pi(x)$ integrable (w.r.t. $\pi(x)dx$) we have:
\[
    \EEE{X\sim\pi} f(X)\grad\log\pi(X)  = -\EEE{X\sim\pi}\grad f(X)
\]
\end{theorem}

\begin{remark}[Relationship to Stein's lemma and Stein's method]
Similar results are common in the literature of Stein's method, (this result is even sometimes called Stein's lemma - at least in the case $\pi$ is the density of a normal distribution). For example, a similar result is found in \citep{gorham2017measuring}, From what we have seen, the literature in that area considers only bounded integrands for general densities or absolutely integrable integ\-rands for specific densities, like the normal. Those results are not sufficient for our purposes here, since we are interested in cases when, for example $f= \grad\log\pi$ which is typically unbounded when $\pi$ is a density. Though our proof is more complicated than that in \citep{gorham2017measuring}, we do accommodate a larger class of integrands (and a larger class of densities as well). The author of this work is not aware of a proof of a distributional integration by parts result which holds as broadly as what we have proven here, without explicitly assuming that boundary integrals tend to $0$ (which would make the results essentially tautological).
\end{remark}

\begin{remark}[On the differentiability of $\pi$ and $f$] By Rademacher's The\-orem (\citep{federer1969geometric}, Theorem 3.1.6), any such $\pi$ or $f$ will be differentiable almost everywhere with a measurable gradient, thus giving meaning to the subsequent expressions.

The assumption that $f$ is locally Lipschitz is equivalent to the assumption that $f$ is Lipschitz on compact sets.
\end{remark}

\begin{remark}[Jacobians and divergences of vector valued functions] By applying the integration by parts formula for real valued functions to each coordinate, the analogous formula for Jacobians also holds. If $\pi$ satisfies the cond\-itions above and $f:\RR^k\to\RR^m$ is locally Lipschitz then:
\[
	\EEE{X\sim\pi}[f(X) \grad\log\pi(X)'] = - \EEE{X\sim\pi}[J f(X)]
\] 
where $J f$ denotes the Jacobian matrix of $f$. Moreover, if $f:\RR^k\to\RR^k$ is locally Lipschitz, since $\divergence(f) = \tr(Jf)$, we have that
\[
	\EEE{X\sim\pi}[\grad\log\pi(X)'f(X)] 
		& = \tr( \EEE{X\sim\pi}[f(X) \grad\log\pi(X)'])\\
		& = - \tr(\EEE{X\sim\pi}[J f(X)]) \\
		& = - \EEE{X\sim\pi}[\divergence(f)(X)]
\]
\end{remark}

\begin{myproof}[Proof of Distributional Integration by Parts]
By Lemma \ref{lem:char-comp-lip}, there exists a strictly increasing and $\Ll$-almost-everywhere differentiable function $\rho:[0,\infty]\to \RR$ such that $\pi$ is $\rho$-compositionally Lipschitz (see Definition \ref{def:comp-lip}). 

For each $\epsilon>0$, let $A_\epsilon = \pi^{-1}((\epsilon,\infty))  \subset K_\epsilon$. 

By definition we have:
\[
    \EEE{X\sim\pi}\grad f(X)
        & = \int_{\RR^k} \pi(x) \grad f (x)\ d \Ll^k x\\
        & = \int_{\RR^k} \int_0^{\pi(x)}d \Ll^1 r\  \grad f (x)\ d \Ll^k x
\]
Using Fubini-Tonelli's theorem (justified by the assumption that $\grad f(X)$ is integrable w.r.t. $\pi(x)dx$).
\[
	\int_{\RR^k} \int_0^{\pi(x)}d \Ll^1 r\  \grad f (x)\ d \Ll^k x
        \label{eqn:ibps-fubini}       
        & = \int_0^\infty \int_{A_r} \grad f(x)\ d \Ll^k x\ d \Ll^1 r
\]
Using Lemma \ref{lem:gg-spec} we get:
\[
	\int_0^\infty \int_{A_r} \grad f(x)\ d \Ll^k x\ d \Ll^1 r
        & = \int_0^\infty \int_{\partial(A_r)} f(x) \hat n_x\ d\Hh^{k-1}x \ d \Ll^1 r \]   
where $\hat n_x$ is the unit outward-facing normal vector to $\partial(A_r)$. 

If $r\not\in E_\pi = \set{s>0 :\partial(\pi^{-1}((s,\infty))) \neq \pi^{-1}(\set{s})} $, then $\partial(A_r) = \pi^{-1}(\set{r})$, thus (using Lemma \ref{lem:lvl-cts}) for $\Ll^1$-almost every $r>0$:
\[
	\int_{\partial(A_r)} f(x) \hat n_x\ d\Hh^{k-1}x 
		& = \int_{\pi^{-1}(\set{r})} f(x) \tilde n_x\ d\Hh^{k-1}x
\] 
where $\tilde n_x$ is the unit outward-facing normal vector to $\pi^{-1}(\set{r})$ when $r\not\in E_\pi$ and is $0$ otherwise.  Therefore:
\[
	\int_0^\infty \int_{\partial(A_r)} f(x) \hat n_x\ d\Hh^{k-1}x \ d \Ll^1 r
	\label{eqn:ibps-grad=normal}
        & = \int_0^\infty \int_{\pi^{-1}(\set{r})} f(x) \tilde n_x \ d\Hh^{k-1}x\ d \Ll^1 r		
\]
Changing from $\pi$-coordinates to $\rho\circ \pi$ coordinates, $s = \rho(r)$, so that our level sets are taken with respect to a Lipschitz function (and hence we can later apply the coarea formula):
\[
    & \int_0^\infty \int_{\pi^{-1}(\set{r})} f(x) \tilde n_x \ d\Hh^{k-1}x\ d \Ll^1 r	\\
        & = \int_{\rho(0)}^{\rho(\infty)} \frac{1}{\rho'(\rho^{-1}(s))} \int_{(\rho\circ\pi)^{-1}(\set{s})} f(x) \tilde n_x \ d\Hh^{k-1}x\ d \Ll^1 s \\
        & = \int_{\rho(0)}^{\rho(\infty)} \int_{(\rho\circ\pi)^{-1}(\set{s})} \frac{f(x)}{\rho'(\pi(x))} \tilde n_x \ d\Hh^{k-1}x\ d \Ll^1 s \\
\]

Now, applying the co-area formula \citep{federer1969geometric} (Theorem 3.2.12 therein) coordinate-wise,  we get:  
\[
	& \int_{\rho(0)}^{\rho(\infty)} \int_{(\rho\circ\pi)^{-1}(\set{s})} \frac{f(x)}{\rho'(\pi(x))} \tilde n_x \ d\Hh^{k-1}x\ d \Ll^1 s \\
        & = \int_{\RR^k} f(x) \frac{1}{\rho'\circ\pi(x)} \tilde n_x \  J_1[\rho\circ\pi(x)]\ d \Ll^k x
\]
Here $J_1 [\rho\circ\pi]$ denotes the ``$1\times 1$ Jacobian of $\rho\circ \pi$'' \citep{federer1969geometric} (Definition 3.2.1 therein). The formula $J_1\pi(x) = \norm{\wedge_1 \grad\rho\circ\pi(x)} = \rho'(\pi(x)) \norm{\grad\pi(x)}$ simplifies our expression to:
\[
	\int_{\RR^k} f(x) \frac{1}{\rho'\circ\pi(x)} \tilde n_x \  J_1[\rho\circ\pi(x)]\ d \Ll^k x
        \label{eqn:ibps-jacobian}  
        & = \int_{\RR^k} f(x) \tilde n_x \norm{\grad\pi(x)} d \Ll^k x
\]
Next, since when $\grad\pi(x)\neq 0$ then the negative standardized gradient is the outward facing normal to the level set (and when $\grad\pi(x) = 0$ both integrands will be $0$):
\[
    \int_{\RR^k} f(x) \tilde n_x \norm{\grad\pi(x)} d \Ll^k x
        & = -\int_{\RR^k} f(x) \grad\pi(x) d \Ll^k x
\]
Further simplification yields:
\[
	-\int_{\RR^k} f(x) \grad\pi(x)d \Ll^k x
        & = -\int_{\RR^k} f(x) \grad\log\pi(x) \pi(x) d \Ll^k x \\
        & = -\EEE{X\sim\pi} f(X)\grad\log\pi(X)         
\] 
\end{myproof}

\begin{definition}[Compositionally Lipschitz function] \label{def:comp-lip}
If $f$ is a function on $\RR^k$ and there exists a strictly increasing and $\Ll^1$ almost everywhere differ\-entiable function $\rho:[0,\infty]\to \RR$ such that $\rho\circ f$ is Lipschitz, then we will call $f$ \textit{compositionally Lipschitz}. 
\end{definition}

If we wish to emphasise the $\rho$ used we may use the term $\rho$-\textit{compositionally Lipschitz}, and when we wish to emphasise the $\rho$ used and the Lipschitz constant of $\rho\circ f$ we may use the term $(\rho,L)$-\textit{compositionally Lipschitz}. 

Of course, if $\rho$ is the identity function then the function is just Lipschitz. More generally, if $\rho^{-1}$ is $L_1$-Lipschitz and $f$ is $(\rho,L_2)$-compositionally Lip\-schitz then $f$ is also $(L_1 L_2)$-Lipschitz.

\begin{lemma}[Characterisation of Compositionally Lipschitz Functions]\label{lem:char-comp-lip}
A continuous, $\Ll^k$-almost-everywhere differentiable differentiable function, $f$, is compositionally Lipschitz if and only if $\norm{\grad f}$ is $\Hh^{k-1}$-essentially bounded on $\Ll^1$-almost-every level set of $f$.
\end{lemma}

\begin{proof}
$(\Rightarrow)$ Suppose that $f$ is $(\rho,L)$-compositionally Lipschitz. Then (at any point, $x$, where f is differentiable): 
\[
	\norm{\grad[\rho\circ f](x)}
		& = \rho'(f(x)) \norm{\grad f(x)} \leq L \\
\]
Hence:
\[
	\esssup_{x\in f^{-1}(\set{r})} \norm{\grad f(x)}
		& \leq \frac{L}{\rho'(r)}
\]

$(\Leftarrow)$ Suppose that $\norm{\grad f}$ is bounded on $\Ll^1$-almost-every level set of $f$. Then we may take 
\[
	\rho(r) 
		& = \int_0^r\frac{1}{\esssup_{x\in f^{-1}(s)}\norm{\grad f(x)} +1} ds
\]
and we have $0< \norm{\grad[\rho\circ f](x)} = \frac{\norm{\grad f(x)}}{\esssup_{y\in f^{-1}(\set{f(x)})}\norm{\grad f(y)} +1}\leq 1$ so $\rho\circ f$ is $1$-Lipschitz.
\end{proof}

\begin{lemma}[Compositionally Lipschitz Densities have Compact Superlevel Sets]
\label{lem:comp-lvl-1}
If $\pi$ is a continuous $\rho$-compositionally Lipschitz probability density on $\RR^k$ (with Lipschitz constant $L$ for $\rho\circ\pi$) then the superlevel sets of $\pi$, $K_\epsilon = \pi^{-1}([\epsilon, \infty))$, are compact for all $\epsilon>0$.
\end{lemma}

\begin{proof}
Since $\pi$ is continuous, $\pi^{-1}((-\infty,\epsilon))$ is open for all $\epsilon>0$ and so, $K_\epsilon = \pi^{-1}([\epsilon,\infty])$ is closed for all $\epsilon>0$.

Suppose, for contradiction, that $K_\epsilon$ is not compact for some $\epsilon>0$.  For $K_\epsilon$ to fail to be compact, it must be unbounded (since we know it is closed). 

Select $R>0$ with $\frac{\rho(\epsilon)-\rho(\epsilon/2)}{L} >R$. Then, since $K_\epsilon$ is unbounded, we may find $\set{x_j}_{j\in\NN} \subset K_\epsilon$ such that for $i\neq j$, $\norm{x_i - x_j}\geq R$.

Then 
\[
	1 
		& = \int \pi(x) dx \\
		& \geq \sum_{j\in \NN} \int_{B_R(x_j)} \pi(x)dx\\
		& \geq \sum_{j\in \NN} \int_{B_R(x_j)} \rho^{-1}(\rho\circ\pi(x_j) - L\norm {x-x_j}) dx \\
		& \geq \sum_{j\in \NN} \int_{B_R(x_j)} \rho^{-1}(\rho(\epsilon) - L\norm {x-x_j}) dx \\		
		& \geq \sum_{j\in \NN} \int_{B_R(x_j)} \rho^{-1}(\rho(\epsilon/2)) dx \\		
		& =\sum_{j\in\NN} \vol(B_R(0))\ \epsilon/2
\]
This is a contradiction, since the last term is clearly $+\infty$
\end{proof}

\begin{lemma}[Almost every level set of a compositionally Lipschitz density is the boundary of a superlevel set]\label{lem:lvl-cts}
If $\pi$ is a continuous, $\rho$-compositionally Lipschitz probability density (with Lipschitz constant $L$ for $\rho\circ\pi$) then $E_\pi = \set{r>0 :\partial(\pi^{-1}((r,\infty))) \neq \pi^{-1}(\set{r})}$ is countable (and hence has $\Ll^1(E_\pi) = 0$) 
\end{lemma}

\begin{proof}
Since $\pi$ is continuous, $A_r = \pi^{-1}((r,\infty))$ is open. Suppose that $x\in\partial(A_r)$. Then $x\not\in A_r$ so $\pi(x) \leq r$ and $\pi(x)$ is a limit point of $\pi(A_r)=(r,\infty)$. Hence $\partial(A_r)\subset \pi^{-1}(\set{r})$. 

Let $G_\pi = \set{(x,y): x\in\RR^k \andT 0\leq y\leq \pi(x)}$. Then $\Ll^{k+1}(G_\pi) = 1$.

Suppose $r\in E_\pi$. Then there exists an $x_r\in \pi^{-1}(\set{r})\setminus\partial(A_r)$ with $\delta_r = d(x_r,A_r) \wedge \frac{\rho(r)-\rho(r/2)}{L}>0$. Since $\rho\circ \pi$ is Lipschitz, the $\rho$-transformed cone:
\[
    C_r 
        & = \lcr\{{(x,y) \st x\in B_{\delta_r}(x_r)}.\\
        &\qquad\qquad \lcr.{\andT (\rho\circ\pi(x_r)-\delta_r L)\leq \rho(y)\leq (\rho\circ\pi(x_r)-d(x,x_r)L)}\}
\] 
has $C_r\subset G_\pi$ and $\Ll^{k+1}(C_r)>0$. Also, for $s,r\in E_\pi$ with $s<r$ we have $C_s\cap C_r = \emptyset$. Hence $\sum_{r\in E_\pi} \Ll^{k+1}(C_r) \leq \Ll^{k+1}(G_\pi) = 1$. This implies that $E_\pi$ is at most countable.
\end{proof}

\begin{lemma}[Specialised Gauss-Green Theorem]
\label{lem:gg-spec} If $\pi$ is a probability dens\-ity on $\RR^k$, and $\pi$ is $\rho$-compositionally Lipschitz, then for any $f$ which is locally Lipschitz and for any $r>0$
\[
	\int_{A_r} \grad f(x)\ d \Ll^k x\
        \label{eqn:ibps-gg}
        & = \int_{\partial(A_r)} f(x) \hat n_x\ d\Hh^{k-1}x \ \]  
where $A_r = \pi^{-1}((r,\infty))$ and $\hat n_x$ is the unit outward facing normal vector to $A_r$ at point $x$.
\end{lemma}
\begin{proof}
This is a specialisation of the Gauss-Green theorem from geometric measure theory \citep{federer1969geometric} (Corollary 4.5.6 therein) to the problem at hand. To apply the cited form of the theorem, we need to (i) verify that $\partial(\mb E^n \llcorner A_r)$ is representable by integration (a technical condition stated as in \citep{federer1969geometric}), (ii) replace our integrand by one which is Lipschitz instead of locally Lipschitz, and (iii)
convert the result in \citep{federer1969geometric} from a statement about divergences to a statement about gradients.

(i) The use of the Gauss green theorem for these particular domains is justified by the fact that the superlevel sets of Lipshitz functions have locally finite perimeter \citep{maggi2012sets} (Remark 9.5 and Example 12.6 therein), which is sufficient to ensure that $\partial(\mb E^n \llcorner A_r)$ is representable by integration \citep{federer1969geometric} (4.5.12 and 2.10.6 therein). This can be used since the superlevel sets of $\pi$ are also superlevel sets of $\rho\circ \pi$, which is Lipschitz.

(ii) Let $K=K_r$ and let $L$ be the lipschitz constant for $f\vert_{K}$. Define the Whitney and McShane solutions to the Lipschitz extension problem for $(f\vert_K,K)$ as $\ol f(x) = \inf_{y\in K} (f(y)+L \norm{x-y})$ and $\ul f(x) = \sup_{y\in K} (f(y)-L \norm{x-y})$ respectively (see for example \citep{oberman2008explicit}, \citep{whitney1934analytic}, and \citep{mcshane1934extension}). These are the pointwise largest and smallest $L$-Lipschitz extensions of $f\vert_K$. Then a compactly supported and $L$-Lipschitz extension of $f\vert_K$ to $\RR^k$ is given by $f_o(x) := [0\wedge \ol f ](x) + [0\vee \ul f](x)$ (in fact this is the $L$-Lipschitz extension which is pointwise-closest to $0$). 

To see that $f_o$ is an extension of $f\vert_K$ notice that for $x\in K$ if $f(x)>0$ then $f_o(x) = \ul f(x) = f(x)$ or if $f(x)<0$ then $f_o(x) = \ol f(x) = f(x)$ or if $f(x) = 0$ then $f_o(x) = \ol f(x) +\ul f(x) = 2f(x) = 0$. To see that $f_o$ is $L$-Lipschitz, note that $\ul f(x)$ and $\ol f(x)$ are both $L$-Lipschitz, hence so are $[0\wedge \ol f ]$ and $[0\vee \ul f]$. The functions $[0\wedge \ol f ]$ and $[0\vee \ul f]$ have disjoint support, since $\ol f \geq \ul f$ pointwise, so at any $x>0$ either $\ol f(x) <0$ or $\ul f(x)>0$, but not both. To see that $f_o$ is compactly supported, we just need to notice that $\ol f(x) > 0$ and $\ul f(x)<0$ for all $x$ such that $d(x,K)\geq \frac{\sup_{y\in K} \abs{f(y)}}{L}+1 =:R_o$, so $\supp(f_o) \subset K^{R_o}$ where $K^{R_o}$ is the closed $R_o$ neighbourhood of $K$.

(iii) Let $h_{v} = v f_o$. This function is Lipschitz and compactly supported, and is equal to $vf$ in $K_r$ , thus we find:
\[
	v\cdot \int_{A_r} \grad f(x) d \Ll^k x 
		& = \int_{A_r} \divergence[h_{v}](x) d \Ll^k x \\
		&\hspace{-1.7em} \myeq{Gauss-Green} \int_{\partial{A_r}} h_{v}\cdot \hat n_x \ d \Hh^{k-1} x
		 = v \cdot \int_{\partial{A_r}} f(x) \hat n_x \ d \Hh^{k-1} x		
\]
Since this holds for any constant vector $v$, we must have $\int_{A_r} \grad f(x)\ d \Ll^k x  = \int_{\partial{A_r}} f(x) \hat n_x \ d \Hh^{k-1} x$.
\end{proof}

\subsection{Properties of Grad-Log-Lipschitz Densities}
\label{subsec:gll-dens}
A probability density, $\pi$, for which $\grad\log\pi$ is Lipschitz will be referred to as a  grad-log-Lipschitz density.

\begin{lemma}[Grad-Log-Lipschitz Densities are Tangentially Minorised by Gaussians]
\label{lem:tan-min}
If $\pi$ is a probability density on $\RR^k$, and $\grad\log\pi$ is $L$-Lipschitz for each $x_0 \in \RR^k$:
\[
	\pi(x) 
		& \geq \pi(x_0) e^{\norm{\grad\log\pi(x_0)}^2/2L} \exp\lcr({-\norm{x-x_0 - \frac{\grad\log\pi(x_0)}{L}}^2 \frac{L}{2}})\\
	 	& \geq \pi(x_0) \exp\lcr({-\norm{x-x_0 - \frac{\grad\log\pi(x_0)}{L}}^2 \frac{L}{2}}) 
\]
for any $x\in\RR^k$. 

\end{lemma}

\begin{proof}
Since $\grad\log\pi$ is $L$-Lipshitz, 
\[\log\pi(x)-\log\pi(x_0) -(x-x_0)'\grad\log\pi(x_0) \geq - L \norm{x-x_0}^2 /2\]
The result follows by completing the square and exponentiating.
\end{proof}

\begin{lemma}[Grad-Log-Lipschitz Densities are Bounded Above]
\label{lem:bound-abv}
If $\pi$ is a probability density on $\RR^k$, and $\grad\log\pi$ is $L$-Lipschitz then:
\[
	\pi(x)\leq \lcr({\frac{L}{2\pi}})^{k/2} e^{-\norm{\grad\log\pi(x)}^2/2L} \leq\lcr({\frac{L}{2\pi}})^{k/2}
\]
\end{lemma}

\begin{proof}
Using Lemma \ref{lem:tan-min}
\[
	1 
		& = \int \pi(y) dy \\
		& \geq \pi(x)e^{\norm{\grad\log\pi(x)}^2/2L} \int \exp\lcr({-\norm{y-x - \frac{\grad\log\pi(x)}{2L}}^2 \frac L 2}) dy \\
		& = \pi(x)e^{\norm{\grad\log\pi(x)}^2/2L} \int \exp\lcr({-\norm{y}^2 \frac L 2}) dy \\
		& = (2\pi /L)^{k/2} e^{\norm{\grad\log\pi(x)}^2/2L}\pi(x) 		
\]
\end{proof}

\begin{lemma}[Grad-Log-Lipschitz Densities are Lipschitz] \label{lem:lip-gll}
If $\pi$ is a prob\-ability density on $\RR^k$, and $\grad\log\pi$ is $L$-Lipschitz then $\pi$ is $\sqrt{L} e^{-1/2} \lcr({\frac{L}{2\pi}})^{k/2} $-Lipschitz.
\end{lemma}

\begin{proof}
Applying Lemma \ref{lem:bound-abv}, for any $x\in\RR^k$
\[
\norm{\grad\pi(x)} 
  & = \pi(x) \norm{\grad\log\pi(x)} \\
  &\leq \lcr({\frac{L}{2\pi}})^{k/2} e^{-\norm{\grad\log\pi(x)}^2/2L} \norm{\grad\log\pi(x)} \\
  &\leq \lcr({\frac{L}{2\pi}})^{k/2} \sup_{s\geq 0} \lcr({s\ e^{-s^2/ 2L} })
\]

Now, $\frac{d}{ds} \lcr({s\ e^{-s^2/ 2L}}) = \lcr({e^{-s^2/ 2L} (1 - s^2/L)})$ so the maximum of $s\ e^{-s^2/ 2L}$ over $s\geq 0$ occurs at $s = \sqrt{L}$ (since the derivative is positive to left and negative to the right of this value) and the maximum value is $\sqrt{L} e^{-1/2} $.

Hence $\norm{\grad\pi(x)} \leq \sqrt{L} e^{-1/2} \lcr({\frac{L}{2\pi}})^{k/2}$
\end{proof}

\begin{corollary}\label{cor:ibps-gll}
Suppose that $\pi$ is a grad-log-Lipschitz probability density on $\RR^k$. Then, for any $f: \RR^k\to \RR$ which is locally Lipschitz, with $\grad f(X)$ and $f(x)\grad\log\pi(x)$ integrable (w.r.t. $\pi(x)dx$) we have:
\[
    \EEE{X\sim\pi} f(X)\grad\log\pi(X)  = -\EEE{X\sim\pi}\grad f(X)
\]
Similar formulas for the Jacobian and divergence also hold.
\end{corollary}
\begin{proof}
This is just the combination Theorem \ref{thm:ibps} and Lemma \ref{lem:lip-gll}.
\end{proof}

\begin{lemma}\label{lem:gll-glp-var}
\[\EEE{X\sim\pi}[\grad\log\pi(X)] = 0\]
and 
\[\VVar{X\sim\pi}[\grad\log\pi(X)] = - \EEE{X\sim\pi}[\grad^2 \log\pi(X)] \]
\end{lemma}
\begin{proof}
This is an immediate consequence of the integration by parts form\-ula for $\pi$, since $\EE \grad\log\pi = \EE\grad 1 = 0$ and since (using the Jacobian version of the formula, and since the Jacobian of a gradient is a Hessian) 
\[
	\EE\grad\log\pi\grad\log\pi' = -\EE \grad^2\log\pi \ .
\]
\end{proof}

\begin{lemma}
\label{lem:gll-subg}
Let $\pi$ be a $C^2$ probability density on $\RR^k$ such that $\grad\log\pi$ is $L$-Lipshitz. If $X\sim\pi$ then $\grad\log\pi(X)$ is sub-Gaussian with proxy-variance $L$:
\[	\psi(t):= \EEE{X\sim\pi} \exp(\inner{t}{\grad\log\pi(X)}) \leq \exp(L \norm{t}^2 /2)
\]
\end{lemma}

\begin{proof}
Since $\grad\log\pi$ is $L$-Lipshitz, then we must have that $\norm{\grad^2 \log\pi(x)}\leq L $ (and hence $\abs{\laplace \log\pi(X)}\leq kL$ as well) for all $x\in\RR^k$. Let $\mu_n$ denote the $n$th moment of $\norm{\grad\log\pi(X)}$.
\[
	\EEE{X\sim\pi}\lcr[{\norm{\grad\log\pi(X)}^{2}}]
		& = \EEE{X\sim\pi}\lcr[{\grad\log\pi(X) ' \grad\log\pi(X)}]\\
		& = -\EEE{X\sim\pi}\lcr[{\laplace\log\pi(X)}]\\
		& \leq kL		
\]
For $r\geq 2$,
\[
	\mu_{2r} 
		& = \EEE{X\sim\pi}\lcr[{\norm{\grad\log\pi(X)}^{2r}}]\\
		& = \EEE{X\sim\pi}\lcr[{\lcr({\grad\log\pi(X) ' \grad\log\pi(X)})^{r}}]\\
		& = \EEE{X\sim\pi}\lcr[{\lcr({\grad\log\pi(X) ' \grad\log\pi(X)})\lcr({\grad\log\pi(X) ' \grad\log\pi(X)})^{r-1}}]
\]
Using Corollary \ref{cor:ibps-gll},
\[
	&\mu_{2r}		\\
		& = -\EEE{X\sim\pi}\lcr[{\grad\log\pi(X)'\ 2(r-1) \norm{\grad\log\pi(X)}^{2(r-2)} \grad^2\log\pi(X) \grad\log\pi(X)}]\\
		& \qquad - \EEE{X\sim\pi}\lcr[{\lcr({\laplace\log\pi(X)})\norm{\grad\log\pi(X)}^{2(r-1)}}]\\	
		& \leq (2 L (r-1) + kL) \EEE{X\sim\pi}\lcr[{\norm{\grad\log\pi(X)}^{2(r-1)}}]\\
		& = (2 L (r-1) + kL) \mu_{2(r-1)} \\ 
		& \leq (k+2) L r \mu_{2(r-1)} 	 			
\]
Thus $\mu_{2r} \leq r![(k+2) L]^r $. Hence, from \citep{boucheron2013concentration} (Theorem 2.1 therein), we get that $\norm{\grad\log\pi(X)}$ is subgaussian with proxy variance $4 (k+2) L$. Thus we know that $\grad\log\pi(X)$ must also be a subgaussian vector, with proxy variance no larger than $4 (k+2) L$. Now, fix $t\in\RR^k$. The moment generating function of $\grad\log\pi(X)$ is finite everywhere and is given by:
\[
 \psi(t) := \EEE{X\sim\pi} \lcr[{\exp(\inner{t}{\grad\log\pi(X)})}]
\]

Then, letting $B_t$ be the ball of radius $\norm t$ centred at the origin,
\[
	\grad_t \psi(t) 
		& = \grad_t \EEE{X\sim\pi} \lcr[{\exp(\inner{t}{\grad\log\pi(X)})}] \\
		& = \EEE{X\sim\pi} \lcr[{ \grad_t\exp(\inner{t}{\grad\log\pi(X)})}] \\
		& = \EEE{X\sim\pi} \lcr[{ \grad\log\pi(X) \exp(\inner{t}{\grad\log\pi(X)})}] \\
\]
Using integration by parts,
\[
& \hspace{-2em}
\EEE{X\sim\pi} \lcr[{ \grad\log\pi(X) \exp(\inner{t}{\grad\log\pi(X)})}] \\
		& = \EEE{X\sim\pi} \lcr[{(\grad^2\log\pi(X)t)\exp(\inner{t}{\grad\log\pi(X)})}] \\
		& \in L\psi(t) B_t				
\]

Thus:
\[
	&&\grad_t\log\psi(t) &\in L B_t \\	
	&\implies&\log\psi(t) &\leq L\norm{t}^2/2 \\
	&\implies&\psi(t) &\leq \exp(L\norm{t}^2/2)
\]
\end{proof}
\begin{remark}
\label{rem:gll-moments}
Consequently all the moments of $\grad\log\pi$ exist. Moreover, since $\norm{\grad^2\log\pi }\leq L$, all the moments of $\grad^2\log\pi $ must exist as well. This means that the assumptions in \citep{roberts1997weak}, \citep{neal2006optimal}, \citep{bedard2007weak}, etc. that $\EE\lcr({\frac{\pi'}{\pi}})^8<\infty$ (or similar moment conditions) and $\EE\lcr({\frac{\pi''}{\pi}})^4<\infty$ are redundant once $\frac{\pi'}{\pi}$ is assumed to be Lipshitz.
\end{remark}
\newpage
\bibliographystyle{plainnat}
\bibliography{bib-files/blockiid_main}

\newpage
\appendix
\section{Additional Lemmas for the Proof of Weak Convergence}
%!TEX root = ../blockiid_main.tex
\setlength{\parindent}{0pt}
\setlength{\parskip}{6pt}
\begin{lemma}[Gaussian increments of a function approximate derivatives]\label{lem-nr06-a2}
For any $h\in C_c^\infty$, 
\[
  	\scriptstyle \sup_{x\in F_d} \abs{d \EE_{Z} [h(x^{(1):(r)}+Z^{(1):(r)})-h(x^{(1):(r)})] - \frac{l^2}{2} [I_r\otimes \Lambda]: \grad^2 h(x^{(1):(r)})} 
  		& \myram{d\to \infty} 0
\]
and 
\[
	& \scriptstyle  \sup_{x\in F_d}\sqrt{\log d}\lcr\Vert{d \EE_{Z} [Z^{(1):(r)}(h(x^{(1):(r)}+Z^{(1):(r)})-h(x^{(1):(r)}))]}.\\
	&\qquad\qquad\qquad\qquad \scriptstyle\lcr.{ - l^2 [I_r\otimes \Lambda] \grad h(x^{(1):(r)})}\Vert \\
		& \qquad\qquad \myram{d\to \infty} 0
\]
\end{lemma}
\begin{proof} for some random variable, $\eta$,
\[
&\hspace{-2em} h(x^{(1):(r)}+Z^{(1):(r)})-h(x^{(1):(r)})\\
  & = (\grad h(x^{(1):(r)}))' Z^{(1):(r)} + \frac{1}{2} {Z^{(1):(r)}}'\grad^2 h(x^{(1):(r)}) Z^{(1):(r)} \\
  &\qquad\qquad +\sum_{\abs{\alpha}=3} \frac{1}{\alpha!} D^{\alpha} h(x^{(1):(r)}+\eta Z^{(1):(r)}) (Z^{(1):(r)})^\alpha  
\]
Thus the lemma follows directly by substitution and evaluation, since the third order partials of $h$ must be uniformly bounded,  and since the third moments of $Z^{(1):(r)}$ are $O(d^{-3/2})$.
\end{proof}

\begin{proposition}[Lipschitz Acceptance (\citep{roberts1997weak}, proposition 2.3)]\label{prop-Lip-Acc}
The function $g(x) = 1\wedge e^x$ is $1$-Lipschitz.
\end{proposition}

\begin{proposition}[Acceptance Moments (\citep{roberts1997weak}, proposition 2.4)]\label{prop-censored-LN-mean}
If $W\sim\Nn(\mu,\sigma^2)$ then 
\[
\EE[1\wedge e^W] = \Phi(\mu/\sigma) + e^{\mu+\sigma^2/2}\Phi(-\sigma - \mu/\sigma)
\]
and
\[
\EE[e^W; W<0] = e^{\mu+\sigma^2/2}\Phi(-\sigma - \mu/\sigma)
\]
\end{proposition}

\begin{lemma}\label{lem-Tr-Err-bd}
Let 
\[
	W_d (x) 
		& = \frac{1}{2} \sum_{i = r+1}^d \scriptstyle \lcr[{ {Z^{(i)}} ' [\grad^2 \log \pi(x^{(i)})] Z^{(i)} + \frac{l^2}{(d-1)} (\grad \log\pi(x^{(i)}))' \Lambda (\grad \log\pi(x^{(i)}))}]\\
		& = \frac{1}{2} \sum_{i =  r+1}^d \scriptstyle \lcr[{ [\grad^2 \log \pi(x^{(i)})] : [Z^{(i)}{Z^{(i)}} '] + \frac{l^2}{(d-1)} (\grad \log\pi(x^{(i)}))' \Lambda (\grad \log\pi(x^{(i)}))}]
\]
where $Z^{(i)} \mysim{iid} \Nn\lcr({0,\frac{l^2\Lambda}{(d-1)}})$.

Then $\lim_{d\to\infty} \sup_{x^{(1):(d)}\in F_d} \EE\abs{W_d(x)} = 0$
\end{lemma}

\begin{proof}
Using Isserlis' theorem (\citep{isserlis1916on}, equation (39)*):
\[
& 	\EE\lcr[{([\grad^2 \log \pi(x^{(i)})] : [Z^{(i)}{Z^{(i)}} '])^2}] \\
		& = \EE\lcr[{\lcr({\sum_{\alpha,\beta} \frac{\partial^2 \log\pi (x^{(i)})}{\partial x^{(i)}_\alpha \partial x^{(i)}_\beta}Z^{(i)}_\alpha Z^{(i)}_\beta})^2  }] \\
%		& = \EE\lcr[{\sum_{\alpha,\beta,\gamma,\delta} \frac{\partial^2 \log\pi (x^{(i)})}{\partial x^{(i)}_\alpha \partial x^{(i)}_\beta} \frac{\partial^2 \log\pi (x^{(i)})}{\partial x^{(i)}_\gamma \partial x^{(i)}_\delta} Z^{(i)}_\alpha Z^{(i)}_\beta Z^{(i)}_\gamma Z^{(i)}_\delta  }]\\
		& =\sum_{\alpha,\beta,\gamma,\delta} \frac{\partial^2 \log\pi (x^{(i)})}{\partial x^{(i)}_\alpha \partial x^{(i)}_\beta} \frac{\partial^2 \log\pi (x^{(i)})}{\partial x^{(i)}_\gamma \partial x^{(i)}_\delta}  \EE\lcr[{Z^{(i)}_\alpha Z^{(i)}_\beta Z^{(i)}_\gamma Z^{(i)}_\delta  }]\\
		& =\frac{l^4}{(d-1)^2} \sum_{\alpha,\beta,\gamma,\delta} \scriptstyle \frac{\partial^2 \log\pi (x^{(i)})}{\partial x^{(i)}_\alpha \partial x^{(i)}_\beta} \frac{\partial^2 \log\pi (x^{(i)})}{\partial x^{(i)}_\gamma \partial x^{(i)}_\delta}  \lcr({\Lambda_{\alpha\beta}\Lambda_{\gamma\delta} +\Lambda_{\alpha\gamma}\Lambda_{\beta\delta}+\Lambda_{\alpha\delta}\Lambda_{\beta\gamma}})\\
		& =\frac{l^4}{(d-1)^2}\lcr[{ \lcr({\sum_{\alpha,\beta} {\scriptstyle \frac{\partial^2 \log\pi (x^{(i)})}{\partial x^{(i)}_\alpha \partial x^{(i)}_\beta} \Lambda_{\alpha\beta}}})^2 + 2\sum_{\alpha,\beta,\gamma,\delta} {\scriptstyle \frac{\partial^2 \log\pi (x^{(i)})}{\partial x^{(i)}_\alpha \partial x^{(i)}_\beta} \frac{\partial^2 \log\pi (x^{(i)})}{\partial x^{(i)}_\gamma \partial x^{(i)}_\delta}\Lambda_{\alpha\gamma}\Lambda_{\beta\delta}}}]\\
		& =\frac{l^4}{(d-1)^2} \lcr[{ [\Lambda : \grad^2\log\pi(x^{(i)})]^2 + 2 \sum_{\alpha,\beta,\gamma,\delta} {\scriptstyle \frac{\partial^2 \log\pi (x^{(i)})}{\partial x^{(i)}_\alpha \partial x^{(i)}_\beta} \frac{\partial^2 \log\pi (x^{(i)})}{\partial x^{(i)}_\gamma \partial x^{(i)}_\delta}\Lambda_{\alpha\gamma}\Lambda_{\beta\delta}}}]\\		
\]
where all Greek subscripts range over $\set{1,...,k}$.

Thus we have:
\[
	&\lcr[{\EE\abs{W_d(x)}}]^2\\
		& \leq \EE\lcr[{\abs{W_d(x)}^2}]\\
		& = \frac{l^4}{4(d-1)^2} \lcr({\sum_{i=r+1}^d {\scriptstyle \lcr({[\Lambda : \grad^2\log\pi(x^{(i)})] + (\grad \log\pi(x^{(i)}))' \Lambda (\grad \log\pi(x^{(i)}))  })}})^2\\
		& \qquad + \frac{2 l^4}{4(d-1)^2} \sum_{i=r+1}^d\sum_{\alpha,\beta,\gamma,\delta} \frac{\partial^2 \log\pi (x^{(i)})}{\partial x^{(i)}_\alpha \partial x^{(i)}_\beta} \frac{\partial^2 \log\pi (x^{(i)})}{\partial x^{(i)}_\gamma \partial x^{(i)}_\delta}\Lambda_{\alpha\gamma}\Lambda_{\beta\delta}
\]

For $x^{(1):(d)} \in F_d$ we have: 
\[
\frac{1}{2(d-1)}\abs {\sum_{i=r+1}^d {\scriptstyle \lcr({[\Lambda : \grad^2\log\pi(x^{(i)})] + (\grad \log\pi(x^{(i)}))' \Lambda (\grad \log\pi(x^{(i)}))  })}} \leq d^{-1/8}
\]

Since $\grad\log\pi$ is Lipschitz, the second order partials of $\log\pi$ are essentially bounded, 

hence $\sum_{\alpha,\beta,\gamma,\delta} \frac{\partial^2 \log\pi (x^{(i)})}{\partial x^{(i)}_\alpha \partial x^{(i)}_\beta} \frac{\partial^2 \log\pi (x^{(i)})}{\partial x^{(i)}_\gamma \partial x^{(i)}_\delta}\Lambda_{\alpha\gamma}\Lambda_{\beta\delta}$ is essentially bounded. Thus:
\[
\sup_{x\in\RR^{kd}} \frac{2 l^2}{4(d-1)^2} \sum_{i=r+1}^d\sum_{\alpha,\beta,\gamma,\delta} \frac{\partial^2 \log\pi (x^{(i)})}{\partial x^{(i)}_\alpha \partial x^{(i)}_\beta} \frac{\partial^2 \log\pi (x^{(i)})}{\partial x^{(i)}_\gamma \partial x^{(i)}_\delta}\Lambda_{\alpha\gamma}\Lambda_{\beta\delta} \in O(1/d)
\]
Combining these two limits we get that $\lim_{d\to\infty} \sup_{x\in F_d} \EE\abs{W_d(x)} = 0$.
\end{proof}
\newpage

\end{document}